\documentclass{amsart}

\usepackage[paperheight=11in, 
    paperwidth=8.5in,
    outer=1.2in,
    inner=1.2in,
    bottom=1.4in,
    top=1.4in]{geometry} 

\usepackage{amsmath, amsxtra, amsthm, amsfonts, amssymb}
\usepackage{mathrsfs}
\usepackage{graphicx}
\usepackage{url}
\usepackage{color}
\usepackage{bbm}
\usepackage{tikz-cd}

\usepackage[T1]{fontenc}
\usepackage{enumitem}

 \usepackage[hidelinks,backref=page]{hyperref} 
 \hypersetup{
    colorlinks,
    citecolor=magenta,
    filecolor=magenta,
    linkcolor=blue,
    urlcolor=black
}

\usepackage{cleveref}

\usepackage[alphabetic]{amsrefs}
\usepackage{mathrsfs}

\newcommand{\ZZ}{\mathbb{Z}}
\newcommand{\PP}{\mathbb{P}}
\newcommand{\kk}{\mathbbm{k}}
\newcommand{\RR}{\mathbb{R}}
\newcommand{\QQ}{\mathbb{Q}}

\newcommand{\M}{\mathrm{M}}
\renewcommand{\P}{\mathrm{P}}
\newcommand{\be}{\mathbf{e}}
\newcommand{\E}{{[m]}}
\newcommand{\EE}{E}
\newcommand{\rk}{\operatorname{rk}}

\RequirePackage{xspace}
\RequirePackage{etoolbox}
\RequirePackage{varwidth}
\RequirePackage{enumitem}
\RequirePackage{tensor}
\RequirePackage{mathtools}
\RequirePackage{longtable}
\RequirePackage{multirow}

\theoremstyle{definition}
\newtheorem{theorem}{Theorem}[section]
\newtheorem{corollary}[theorem]{Corollary}
\newtheorem{lemma}[theorem]{Lemma}
\newtheorem{proposition}[theorem]{Proposition}
\newtheorem{definition}[theorem]{Definition}
\newtheorem{conjecture}[theorem]{Conjecture}
\newtheorem{example}[theorem]{Example}
\newtheorem{remark}[theorem]{Remark}

\newtheorem{question}[theorem]{Question}

\newtheorem{notation}[theorem]{Notation}

\title{K-theoretic positivity for matroids}
\author{Christopher Eur, Matt Larson}

\begin{document}
\
\maketitle

\vspace{-21pt}

\begin{abstract}
Hilbert polynomials have positivity properties under favorable conditions.  We establish a similar ``$K$-theoretic positivity'' for matroids.
As an application, for a multiplicity-free subvariety of a product of projective spaces such that the projection onto one of the factors is generically finite onto its image, we show that a transformation of its $K$-polynomial is Lorentzian.
This partially answers a conjecture of Castillo, Cid-Ruiz, Mohammadi, and Monta\~{n}o.
As another application, we show that the $h^*$-vector of a simplicially positive divisor on a matroid
is a Macaulay vector, affirmatively answering a question of Speyer for a new infinite family of matroids.
\end{abstract}

\section{Introduction}

For a $d$-dimensional lattice polytope $Q$, 
Stanley \cite{StanleyEhrhart} showed that the $h^*$-vector $(h^*_0(Q), \dotsc, h^*_d(Q))$ defined by
\[
\sum_{k \geq 0} |\{\text{lattice points in $kQ$}\}|q^k = \frac{h_0^*(Q) + h_1^*(Q) q + \dotsb + h_d^*(Q) q^d}{(1-q)^{d+1}}
\]
is nonnegative, and it is furthermore a Macaulay vector (Definition~\ref{def:Macaulayvector}) if, for every $k$, all lattice points in $kQ$ are sums of lattice points in $Q$.
Via standard results in toric geometry \cite[Chapter 9]{CLS}, this result can be formulated geometrically as  ``$K$-theoretic positivity'' in the following way.

Let $X$ be a smooth projective toric variety with fan $\Sigma$, and 
let $\chi \colon K(X) \to \ZZ$ be the sheaf Euler characteristic map on the Grothendieck ring $K(X)$ of vector bundles on $X$.
For a nef line bundle $\mathcal{L}$ associated to a lattice polytope $Q$ whose normal fan coarsens $\Sigma$,
toric vanishing theorems imply that $\chi(X,\mathcal L^{\otimes k}) = \dim H^0(X,\mathcal L^{\otimes k}) = |\{\text{lattice points in $kQ$}\}|$ (for $k\geq 0$), and that the graded ring
$
R_{\mathcal{L}}^{\bullet} \coloneqq  \bigoplus_{k \ge 0} H^0(X, \mathcal{L}^{\otimes k})
$
is Cohen--Macaulay.
See Proposition~\ref{prop:ampleMacaulay} for a detailed review.
Quotienting $R_{\mathcal L}^{\bullet}$ by a linear system of parameters, the vector $(h_0^*(\mathcal L), \dotsc, h_d^*(\mathcal L))$ defined by
$$\sum_{k \ge 0} \chi(X, \mathcal{L}^{\otimes k}) q^k = \text{Hilbert series of $R_{\mathcal L}^{\bullet}$} = \frac{h_0^*(\mathcal L) + h_1^*(\mathcal L)q + \dotsb + h_d^*(\mathcal L)q^d}{(1 - q)^{\dim Q + 1}}$$
is the Hilbert function of a graded artinian ring.  In particular, 
the vector $(h_0^*(\mathcal L), \dotsc, h_d^*(\mathcal L))$ is nonnegative, and it is furthermore a Macaulay vector if $R_{\mathcal{L}}^{\bullet}$ is generated in degree $1$.

\smallskip
Here, we establish a similar positivity property for matroids.
We begin in the more general setting of polymatroids.
For a nonnegative integer $m$, let $\E =\{1, \dotsc, m\}$, and let $\textbf a = (a_1, \ldots, a_m)$ be a sequence of nonnegative integers.

\begin{definition}
A \emph{polymatroid} $\P$ on $\E$ with \emph{cage} $\textbf a$ is 
a function $\operatorname{rk}_\P \colon 2^\E \to \mathbb{Z}_{\ge 0}$ satisfying
\begin{enumerate}
\item (Submodularity) $\operatorname{rk}_\P(I_1) + \operatorname{rk}_\P(I_2) \ge \operatorname{rk}_\P(I_1 \cap I_2) + \operatorname{rk}_\P(I_1 \cup I_2)$ for any $I_1, I_2 \subseteq \E$,
\item (Monotonicity) $\operatorname{rk}_\P(I_1) \le \operatorname{rk}_\P(I_2)$ for any $I_1 \subseteq I_2\subseteq \E$,
\item (Normalization) $\operatorname{rk}_\P(\emptyset) = 0$, and
\item (Cage) $\operatorname{rk}_\P(i) \le a_i$ for any $i \in \E$. 
\end{enumerate}
We say that $\operatorname{rk}_\P$ is the \emph{rank function} of the polymatroid $\P$, and that $\P$ has \emph{rank} $r=\operatorname{rk}_\P(\E)$.
\end{definition}

A \emph{matroid} is a polymatroid with cage $(1, \dotsc, 1)$.
See \cite{Wel76} for the fundamentals of matroid theory.
In \cite{LLPP}, analogues of $K$-rings for matroids were introduced, modeled after the following geometry of realizable matroids. Let $\kk$ be a field. 
A \emph{realization} of a matroid $\M$ on a finite set $\EE$ is a linear subspace $L \subseteq \kk^{\EE}$ such that
$
\operatorname{rk}_\M(S) = \dim\big(\text{image of $L$ under the projection $\kk^{\EE} \to \kk^S$}\big)
$
for all $S\subseteq \EE$.
A realization $L \subseteq \kk^{\EE}$ defines a smooth projective irreducible variety $W_L$ called the \emph{augmented wonderful variety} \cite{BHMPW20a}, defined by
\[
W_L = \text{the closure of the image of $L$ in $\prod_{\emptyset \subsetneq S \subseteq \EE} \mathbb{P} \big(\kk^S \oplus \kk\big)$},
\]
where the map $L \to \mathbb{P}(\kk^S \oplus \kk)$ is the composition of the projection $L\to \kk^S$ with the projective completion $\kk^S \hookrightarrow \PP(\kk^S \oplus \kk)$.
For $\emptyset \subsetneq S\subseteq E$, 
let $\mathcal L_S$ be the line bundle on $W_L$ obtained by pulling back $\mathcal O(1)$  from $\PP(\kk^S \oplus \kk)$.
These line bundles $\{\mathcal L_S\}_{\emptyset\subsetneq S \subseteq E}$ generate the Picard group of $W_L$, and their $K$-classes $\{[\mathcal L_S]\}_{\emptyset\subsetneq S \subseteq \EE}$ generate the Grothendieck ring of vector bundles $K(W_L)$ as a ring \cite[Theorem 5.2]{LLPP}.

\medskip
For an arbitrary (not necessarily realizable) matroid $\M$, the authors of \cite{LLPP} introduced the \emph{augmented $K$-ring} $K(\M)$ of $\M$.  The following are its key properties:
\begin{enumerate}[label = (\roman*)]
\item It is equipped with an ``Euler characteristic map'' $\chi(\M,-) \colon K(\M) \to \ZZ$.
\item Each nonempty subset $S\subseteq \EE$ defines an element $[\mathcal L_S] \in K(\M)$ such that $\{[\mathcal L_S]\}_{\emptyset \subsetneq S \subseteq \EE}$ generates $K(\M)$ as a ring.  A \emph{line bundle} in $K(\M)$ is a Laurent monomial in the $[\mathcal L_S]$.
\item When $\M$ has a realization $L\subseteq \kk^{\EE}$, identifying the $[\mathcal L_S]$ in $K(\M)$ and $K(W_L)$ gives an isomorphism $K(\M) \simeq K(W_L)$ such that  $\chi(\M,-) = \chi(W_L, -)$.
\end{enumerate}
See Section~\ref{ssec:augK} for the definition of $K(\M)$ and further properties of of $K(\M)$ and $\chi(\M,-)$.

\medskip
To state our main theorem about $K(\M)$, we prepare with the following constructions: 
\begin{itemize}
\item For a matroid $\M$ on a finite set $\EE$ and subsets $S_1, \dotsc, S_m\subseteq \EE$, the function $\rk \colon 2^{\E} \to \ZZ$ defined by $\rk(I) = \rk_\M(\bigcup_{i\in I} S_i)$ is a polymatroid, which we call the \emph{restriction polymatroid of $\M$ to $S_1, \dotsc, S_m$}.  Every polymatroid is a restriction polymatroid of a matroid; see Definition~\ref{defn:lift}.
\item For a polymatroid $\P$ with cage $(a_1, \dotsc, a_m)$, define a subvariety $Y_\P \subseteq \PP^{a_1} \times \cdots \times \PP^{a_m}$ as follows.  For $i\in [m]$ and an integer $0\leq j \leq a_i$, let $L_i(j)$ be the $j$-dimensional linear subvariety $\{[x_0, \dots, x_{a_i}] \in \PP^{a_i}: x_k = 0 \text{ if $k > j$}\}$ of $\PP^{a_i}$.  We define
\[
Y_\P = \bigcup_{(b_1, \dotsc, b_m)\in B(\P)} L_1(b_1) \times \dotsb \times L_m(b_m)
\]
where the union runs over all lattice points $(b_1, \dotsc, b_m) \in \ZZ^m$ in the \emph{base polytope} of $\P$ defined as $B(\P) = \{ (x_1, \dotsc, x_m) \in \RR^m_{\geq 0} : \sum_{i\in [m]} x_i = \operatorname{rk}_\P([m]) \text{ and } \sum_{i\in I} x_i \leq \operatorname{rk}_\P(I) \text{ for all $I\subseteq [m]$}\}.$
Note that the variety $Y_\P$ and the restrictions to $Y_\P$ of the line bundles $\mathcal{O}(k_1, \dotsc, k_m)$ on $\PP^{a_1} \times \dotsm \times \PP^{a_m}$  does not depend on the choice of the cage $\textbf a$. 
\end{itemize}

\begin{theorem}\label{thm:projection}
For a polymatroid $\P$ and a matroid $\M$ on $\EE$ with subsets $S_1, \dotsc, S_m \subseteq \EE$ such that the restriction polymatroid is $\P$, one has
\[
\chi(\M, \mathcal L_{S_1}^{\otimes k_1} \otimes \dotsb  \otimes \mathcal L_{S_m}^{\otimes k_m}) = \chi(Y_\P, \mathcal O(k_1, \dots, k_m)) \quad\text{for all $k_1, \dotsc, k_m$}.
\] 
\end{theorem}

This theorem originates from the following geometry.
Let $X \subseteq \PP^{a_1} \times \dotsb \times \PP^{a_m}$ be an irreducible \emph{multiplicity-free subvariety} (i.e., the coefficients of its multidegree are 0 or 1).
The function $\rk_\P \colon  2^{[m]} \to \ZZ$ defined by $\rk_\P(I) = \dim \big(\text{image of $X$ under the projection to $\prod_{i\in I} \PP^{a_i}$}\big)$ is a polymatroid $\P$ by \cite[Corollary 4.7]{BH}. 
Brion \cite{BrionMultiplicity} showed that any such $X$ has a flat degeneration to $Y_\P$.

For example, if a matroid $\M$ has a realization $L\subseteq \kk^E$, and if  the dimension of the projection of $W_L$ to $\PP(\kk^{S_1} \oplus \kk) \times \dotsb \times \PP(\kk^{S_m} \oplus \kk)$ is equal to the dimension of $L$, then the projection is an irreducible multiplicity-free subvariety $X$ whose polymatroid $\P$ is the restriction polymatroid of $\M$ to $S_1, \dotsc, S_m$.
Thus, if furthermore the projection is an isomorphism $W_L\simeq X$, Brion's flat degeneration implies Theorem~\ref{thm:projection} in this special case. 
We prove Theorem~\ref{thm:projection} in general by using properties of \emph{polymatroid valuativity} \cites{DF10,ELPoly} (Definition~\ref{def:valuative}) and the fact that multiplicity-free subvarieties have rational singularities in characteristic $0$ \cite[Theorem 4.3]{BergetFink}.

\medskip
On the other hand, combining Brion's flat degeneration with Theorem~\ref{thm:projection} implies the following.

\begin{corollary}\label{cor:multfree}
Let $\P$, $\M$, and $(S_1, \dotsc, S_m)$ be as above.  If $X \subseteq \PP^{a_1} \times \dotsb \times \PP^{a_m}$ is any irreducible multiplicity-free subvariety with $\operatorname{rk}_\P(I) = \dim\big(\operatorname{Image}(X \to \prod_{i\in I} \PP^{a_i})\big)$ for all $I\subseteq [m]$, then one has 
\[\chi(\M, \mathcal L_{S_1}^{\otimes k_1} \otimes \dotsb  \otimes \mathcal L_{S_m}^{\otimes k_m}) = \chi(X, \mathcal O(k_1, \dots, k_m))
\]
for all line bundles $\mathcal O(k_1, \dotsc, k_m)$ on $\mathbb{P}^{a_1} \times \dotsb \times \mathbb{P}^{a_m}$.
\end{corollary}

As an application, we use Corollary~\ref{cor:multfree} to study Snapper polynomials of multiplicity-free subvarieties via matroid theory.
For a projective variety $X$ and line bundles $\mathcal{L}_1, \dotsc, \mathcal{L}_m$ on $X$, the function assigning to each tuple of integers $(t_1, \dotsc, t_m)$ the Euler characteristic $\chi(X, \mathcal{L}_1^{\otimes t_1} \otimes \dotsb \otimes \mathcal{L}_m^{\otimes t_m})$ is a polynomial \cite{Snapper}, which is often called the \emph{Snapper polynomial}. This property also holds for $\chi(\M, -)$, allowing us to define Snapper polynomials for matroids, for which we establish the following.

\smallskip
For a sequence $\mathbf k = (k_1, \dotsc, k_m)$ of nonnegative integers, set $|\mathbf k| = \sum_i k_i$,  and denote $\mathbf t^{\mathbf k} = t_1^{k_1}\dotsm t_m^{k_m}$ and $\mathbf t^{[\mathbf k]} = \binom{t_1+k_1}{k_1}\dotsm \binom{t_m + k_m}{k_m}$, where $\binom{t}{k} = \frac{t(t-1) \dotsm (t-k+1)}{k!}$.

\begin{theorem}\label{thm:lorentzian}
For a matroid $\M$ on $E$ and subsets $S_1, \dotsc, S_m$ of $E$ whose restriction polymatroid has rank $r$, define a polynomial $H(t_1, \dotsc, t_m)$ by
\[
H(t_1, \dotsc, t_m) = \sum_{\mathbf k} a_{\mathbf k} \mathbf t^{\mathbf k} \quad\text{such that}\quad \chi\Big(\M, \mathcal L_{S_1}^{\otimes t_1} \otimes \dotsm \otimes \mathcal L_{S_m}^{\otimes t_m}\Big) = \sum_{\mathbf k} (-1)^{r-|\mathbf k|} a_{\mathbf k}  \mathbf t^{[\mathbf k]}.
\]
Suppose at least one of $S_1, \dotsc, S_m$ satisfies $\rk_\M(S_i) = r$.
Then, the homogenization $\widetilde H(\mathbf t, t_0) = \sum_{\mathbf k} a_{\mathbf k}\mathbf t^{\mathbf k}  t_0^{r - |\mathbf k|}$ by an auxiliary variable $t_0$ is \emph{denormalized Lorentzian} in the sense of \cite{BH}.
\end{theorem}

Theorem~\ref{thm:lorentzian} positively answers \cite[Conjecture 7.18 and Question 7.21]{CCMM} about the ``twisted $K$-polynomial'' of a multiplicity-free subvariety $X\subseteq \PP^{a_1} \times \dotsm \times \PP^{a_m}$ when, for some $i$, the dimension of the image of the projection $X \to \PP^{a_i}$ is equal to the dimension of $X$. 
In Section~\ref{ssec:appLorent}, we explain how 
\cite[Question 7.21]{CCMM} is equivalent to asking whether Theorem~\ref{thm:lorentzian} remains true without the condition ``at least one of $S_1, \dotsc, S_m$ satisfies $\rk_\M(S_i) = r$,'' which was needed in our proof.
See Remark~\ref{rem:motivation} for further discussion, and 
see Remark~\ref{rem:matroidcase} for another special case of this question. 

\medskip
As another application, we deduce properties of matroids by using geometric properties of $Y_\P$, namely, that $Y_\P$ is Cohen--Macaulay \cite{CCRC} and is a compatibly Frobenius split subvariety of the product of projective spaces \cite{BrionKumar}.
To connect to previous questions in matroid theory, 
it is convenient to phrase our statements in terms of the \emph{non-augmented $K$-ring} $\underline K(\M)$ of a loopless matroid $\M$ on $E$ (see Section~\ref{ssec:nonaug}), for which an analogue of Theorem~\ref{thm:projection} holds (Corollary~\ref{cor:wonderful}).
Like $K(\M)$, the ring $\underline K(\M)$ is equipped with a map $\underline\chi(\M,-) \colon \underline K(\M) \to \ZZ$, and each nonempty subset $S\subseteq E$ defines an element $[\underline{\mathcal L}_S] \in \underline K(\M)$ such that $\{[\underline{\mathcal L}_S]\}$ generates $\underline K(\M)$.
When $\M$ has a realization $L\subseteq \kk^E$, these objects again coincide with those of the (non-augmented) \emph{wonderful variety} $\underline W_L$ of \cite{dCP95}.

\begin{theorem}\label{thm:macaulay}
For a loopless matroid $\M$ and a \emph{line bundle $\mathcal{L}$ in $\underline K(\M)$} (i.e., a Laurent monomial in the $\underline{\mathcal L}_S$), define the \emph{$h^*$-vector} $(h^*_0(\M,\mathcal L), \dotsc, h^*_{d}(\M,\mathcal L))$ to be the coefficients of the polynomial
\[
h^*(\M,\mathcal L; q) = \sum_{k = 0}^{d} h^*_k(\M,\mathcal L) q^k \quad\text{such that}\quad \sum_{k \geq 0}\underline\chi(\M, \mathcal L^{\otimes k}) q^k = \frac{h^*(\M,\mathcal L; q)}{(1-q)^{d+1}}
\]
where $d = \text{degree of the polynomial }\underline\chi(\M, \mathcal L^{\otimes t})$.
If $\mathcal{L}$ is \emph{simplicially positive} (i.e., $\mathcal L = \bigotimes_S \underline{\mathcal L}_S^{\otimes k_S}$ for some nonnegative integers $k_S$), then the {$h^*$-vector} $(h^*_0(\M,\mathcal L), \dotsc, h^*_{d}(\M,\mathcal L))$ is a Macaulay vector and is in particular nonnegative.
\end{theorem}

One verifies that the $h^*$-vector is equivalently defined by the equation
\[
\underline\chi(\M, \mathcal L^{\otimes q}) = \sum_{k = 0}^{d} h^*_k(\M,\mathcal L) \binom{q+d-k}{d},
\]
from which one sees that $(-1)^{d}\underline\chi(\M,\mathcal L^{-1}) = h^*_{d}(\M,\mathcal L)$.
We apply Theorem~\ref{thm:macaulay} in this form to answer affirmatively a question of Speyer \cite{Speyerg} for new infinite families of matroids using a result of Fink, Shaw, and Speyer; see Section~\ref{ssec:gpoly}. 

\smallskip
When $\M$ has a realization $L\subseteq \kk^E$, the simplicially positive line bundles form a full dimensional subcone of the nef cone of $\underline W_L$, but it is usually strictly smaller than the nef cone.
In Section~\ref{ssec:LofP}, we conjecture that the conclusion of Theorem~\ref{thm:macaulay} holds for a larger family of line bundles.
This would answer Speyer's question affirmatively for all matroids. We also establish and conjecture some other properties of $h^*$-vectors of matroids. 

\subsection*{Organization}
In Section~\ref{sec:2}, we recall properties of polymatroids and (augmented) $K$-rings of matroids, and we use them to prove Theorem~\ref{thm:projection}. In Section~\ref{sec:3}, we prove Theorem~\ref{thm:lorentzian} and discuss its consequences.
In Section~\ref{sec:4}, we prove Theorem~\ref{thm:macaulay}. In Section~\ref{sec:applications}, we discuss some applications and some further properties.

\subsection*{Acknowledgements}
We thank Andrew Berget, Dan Corey, Alex Fink, June Huh, Nick Proudfoot, Kris Shaw, and David Speyer for helpful conversations.  We also
thank BIRS for their hospitality in hosting the workshop ``Algebraic Aspects of Matroid Theory.'' We thank the referee for a careful reading and helpful comments. 
The first author is supported by NSF DMS-2246518.
The second author is supported by an ARCS fellowship.

\section{The comparison theorem}\label{sec:2}

We give background on polymatroids in Section~\ref{ssec:polymat}, and we collect properties of the augmented $K$-ring of a matroid in Section~\ref{ssec:augK}.  Then, in Section~\ref{ssec:proofprojection}, we prove Theorem~\ref{thm:projection} comparing the Euler characteristic maps $\chi$ on $K(\M)$ and $Y_\P$.
Analogues for the non-augmented $K$-ring of a matroid are given in Section~\ref{ssec:nonaug}.

\subsection{Polymatroids}\label{ssec:polymat}
We review realizability, valuativity, and lifts for polymatroids.
We begin with realizations. 
Let $\P$ be a polymatroid with cage $(a_1, \dotsc, a_m)$. A \emph{realization} of $\P$ over $\kk$ is a subspace $L \subseteq V_1 \oplus \dotsb \oplus V_m$, where $V_i$ is a vector space over $\kk$ of dimension $a_i$, such that 
$$\operatorname{rk}_P(S) = \dim \left (\text{the image of }L \text{ under the projection to } \bigoplus_{i \in S} V_i \right )$$
for all $S \subseteq [m]$.
When such an $L$ exists, we say $\P$ is \emph{realizable} over $\kk$. When $\P$ is a matroid (i.e., a polymatroid of cage $(1, \dotsc, 1)$), this specializes to realizability of matroids as discussed in the introduction.

\medskip
We will obtain Theorem~\ref{thm:projection} by reducing to the case of realizable matroids.  This reduction step will be facilitated by the notion of valuativity \cites{AFR10,DF10}.

\begin{definition}\label{def:valuative}
For a polymatroid $\P$ on $[m]$, let $\textbf{1}_{\P} \colon \mathbb{R}^m \to \mathbb{Z}$ be the indicator function of its base polytope $B(\P)$. The \emph{valuative group} of polymatroids with cage $\textbf{a} = (a_1, \dotsc, a_m)$, denoted $\operatorname{Val}_{\textbf{a}}$, is the subgroup of $\mathbb{Z}^{(\mathbb{R}^m)}$ generated by $\textbf{1}_\P$ for $\P$ a polymatroid on $[m]$ with cage $\textbf{a}$.

A function from the set of polymatroids with cage $\textbf{a}$ to an abelian group is said to be \emph{valuative} if it factors through $\operatorname{Val}_{\textbf{a}}$.
\end{definition}

By \cite{DF10} or \cite[Remark 3.16]{ELPoly}, $\operatorname{Val}_{\textbf{a}}$ is generated by polymatroids which are realizable over $\mathbb{C}$. In particular, this gives the following useful result.

\begin{corollary}\label{cor:valequal}
Let $f_1$ and $f_2$ be functions from the set of polymatroids with cage $\textbf{a}$ to an abelian group $G$. If $f_1$ and $f_2$ are valuative, and if $f_1(\P) = f_2(\P)$ for any polymatroid $\P$ with cage $\textbf{a}$ that is realizable over $\mathbb{C}$, then $f_1(\P) = f_2(\P)$ for all polymatroids $\P$ with cage $\textbf{a}$. 
\end{corollary}

Lastly, we recall multisymmetric lifts of polymatroids, a construction which has appeared many times in the literature \cites{H72,M75,L77,N86,BCF} with many different names.
We use the terminology and description given in \cites{CHLSW,ELPoly}.

\begin{definition}\label{defn:lift}
Let $\mathrm{P}$ be a polymatroid with cage $\textbf{a} = (a_1, \dotsc, a_m)$ on $[m]$. The \emph{multisymmetric lift} of $\mathrm{P}$ is a matroid $\M$ on a ground set $E$ of size $a_1 + \dotsb + a_m$ which is equipped with a distinguished partition $E = S_1 \sqcup \dotsb \sqcup S_m$ into parts of size $a_1, \dotsc, a_m$ with the following characterizing property: $\operatorname{rk}_{\M}$ is preserved by the action of the product of symmetric groups $\mathfrak{S}_{S_1} \times \dotsb \times \mathfrak{S}_{S_m}$, and
$$\operatorname{rk}_\P(I) = \operatorname{rk}_{\M}\left( \cup_{i \in I} S_i \right) \text{ for all }I \subseteq [m].$$ 
\end{definition}

Note that the multisymmetric lift depends on the choice of cage $\textbf{a}$, and that the restriction polymatroid of the multisymmetric lift $\M$ to the subsets $S_1, \dotsc, S_m$ appearing in the distinguished partition is the polymatroid $\P$.

\smallskip
The construction of the multisymmetric lift respects realizability.
When $\P$ is realized by  $L \subseteq \bigoplus_{i \in [m]} V_i$ over an infinite field $\kk$, the multisymmetric lift of $\P$ can be realized by generically choosing a basis for each $V_i$ to identify $\bigoplus_{i \in [m]} V_i$ with $\kk^{a_1 + \dotsb + a_m}$. 

\subsection{Augmented $K$-rings of matroids}\label{ssec:augK}

Let $\M$ be a matroid on a ground set $E$.

\begin{definition}
The \emph{augmented $K$-ring} $K(\M)$ of $\M$ is the Grothendieck ring of vector bundles on the toric variety $X_{\Sigma_\M}$ of the \emph{augmented Bergman fan} ${\Sigma_{\M}}$ of $\M$.
\end{definition}

The definition of the augmented Bergman fan can be found in \cite{BHMPW20a}, but it won't be needed. In \cite{LLPP}, some some additional structures on $K(\M)$ are constructed. The ring $K(\mathrm{M})$ is equipped with an ``Euler characteristic map'' $\chi(\M, -) \colon K(\mathrm{M}) \to \mathbb{Z}$. Additionally, each nonempty subset $S$ of $\EE$ defines an element $[\mathcal{L}_S] \in K(\M)$. A \emph{line bundle} in $K(\M)$ is a Laurent monomial in the $[\mathcal L_S]$. 
We record the properties of $K(\M)$ that we will need here.

\begin{proposition}\label{prop:augKfeatures}
The augmented $K$-ring $K(\M)$ of $\M$ satisfies the following properties.
\begin{enumerate}[label = (\roman*)]
\item The elements $\{[\mathcal L_S]\}_{\emptyset \subsetneq S \subseteq \EE}$ generate $K(\M)$ as a ring.  
\item When $\M$ has a realization $L\subseteq \kk^{\EE}$, identifying the $[\mathcal L_S]$ in $K(\M)$ and $K(W_L)$ gives an isomorphism $K(\M) \simeq K(W_L)$ such that  $\chi(\M,-) = \chi(W_L, -)$.
\end{enumerate}
\end{proposition}

\begin{proof}
These statements follow from \cite[Theorem 5.2 and Proposition 5.6]{LLPP}.  For (i), the original statement in \cite{LLPP} is in terms of the ${\mathcal L}_F$ for $F$ a nonempty flat of $\M$, but we set $\mathcal L_S = \mathcal L_{\operatorname{cl}_\M(S)}$ where $\operatorname{cl}_\M$ denotes the closure operator of the matroid $\M$.
\end{proof}

We caution that the map $\chi(\M,-)$ is generally different from the sheaf Euler characteristic map $\chi(X_{\Sigma_\M},-)$ of the toric variety $X_{\Sigma_\M}$.
Next, we review a formula for $\chi(\M,-)$ given in \cite{LLPP}, stated in terms of the following definition.

\begin{definition}
We say that a sequence $(S_1, \dotsc, S_m)$ of nonempty subsets of $E$ satisfies the \emph{Hall--Rado} condition (with respect to $\M$) if
\[
\operatorname{rk}_{\M} \left( \bigcup_{i \in I} S_i \right) \ge |I|\quad\text{for every $I \subseteq [m]$}. 
\]
Moreover, we say that $\mathbf k = (k_1, \dotsc, k_m) \in \ZZ_{\geq 0}^m$ satisfies the \emph{Hall--Rado} condition if the sequence $(S_1^{k_1}, \dotsc, S_m^{k_m})$, where $S_i^{k_i}$ denotes $S_i$ repeated $k_i$ times, satisfies the condition, or, equivalently, if
\[
\operatorname{rk}_{\M} \left( \bigcup_{i \in I} S_i \right) \ge \sum_{i \in I} k_i \quad\text{for every $I \subseteq [m]$}. 
\]
\end{definition}

If $\P$ denotes the restriction polymatroid of $\M$ to $S_1, \dotsc, S_m$, note then that $\mathbf k$ satisfies the Hall--Rado condition if and only if it is a lattice point in the \emph{independence polytope of $\P$}, defined as
\[
I(\P) = \{ (x_1, \dotsc, x_m) \in \RR_{\geq 0}^m : \sum_{i \in I} x_i \leq \rk_\P(I) \text{ for all } I \subseteq [m]\}.
\]
The independence polytope $I(\P)$ relates to the base polytope $B(\P)$ by 
$$I(\P) = \{ \mathbf x \in \mathbb{R}^m_{\ge 0} : \mathbf y - \mathbf x \in \RR_{\geq 0}^m \text{ for some } \mathbf y \in B(\P)\},$$
or, equivalently, $I(\P) = (B(\P) + \mathbb{R}^m_{\le 0}) \cap \mathbb{R}^m_{\ge 0}$, where the $+$ denotes Minkowski sum.

\begin{notation}\label{notate:()[]}
To state the formula for $\chi(\M,-)$, it is convenient to introduce the following notation.
For a nonnegative integer $k$, let $t^{(k)}$ denote the polynomial
$\binom{t + k - 1}{d} = \frac{t(t + 1) \dotsb (t + k - 1)}{k!}$.
Recall the previously introduced notation $t^{[k]} = \binom{t+k}{k}$.
The binomial identity $\binom{t+k}{k} = \binom{t+k-1}{k-1}+\binom{t+k-1}{k}$ relates these two notations by $t^{[k]} - t^{[k-1]} = t^{(k)}$, or, equivalently, by $t^{(k)} + t^{(k-1)} + \dotsb + t^{(1)} + 1 = t^{[k]}$.
\end{notation}

\begin{proposition}\label{prop:KHallRado}\cite[Corollary 7.5]{LLPP}
For a sequence $\textbf{k} = (k_S)_{S\in \mathcal S}$ of nonnegative integers indexed by a collection $\mathcal S$ of nonempty subsets of $E$, denote $\textbf{t}^{(\textbf{k})} = \prod t_S^{(k_S)}$.
We have that 
$$\chi(\M, \bigotimes_{S} \mathcal{L}_{S}^{\otimes t_S}) = \sum_{\textbf{k} \text{ satisfies Hall--Rado}} \textbf{t}^{(\textbf{k})}.$$
\end{proposition}

In particular, if $\mathcal{L}$ is a line bundle which is the tensor product of line bundles of the form $\mathcal{L}_{S_i}$ for some subsets $S_1, \dotsc, S_k$ of the ground set of $\M$, then $\chi(\M, \mathcal{L})$ only depends on the restriction polymatroid of $S_1, \dotsc, S_k$. 
We record this observation as the following corollary, which will allow us to reduce the proof of Theorem~\ref{thm:projection} to the case when $\M$ is the multisymmetric lift of $\P$. 

\begin{corollary}\label{cor:indepofM}
Let $\M_1$ and $\M_2$ be matroids, let $S_1, \dotsc, S_m$ be subsets of the ground set of $\M_1$, and let $T_1, \dotsc, T_m$ be subsets of the ground set of $\M_2$. Suppose that the restriction polymatroid of $\M_1$ to $S_1, \dotsc, S_m$ is the same as the restriction polymatroid of $\M_2$ to $T_1, \dotsc, T_m$.
Then, for any $k_1, \dotsc, k_m$,
$$\chi(\M_1, \mathcal{L}_{{S_1}}^{\otimes k_1} \otimes \dotsb \otimes \mathcal{L}_{{S_m}}^{\otimes k_m}) = \chi(\M_2, \mathcal{L}_{{T_1}}^{\otimes k_1} \otimes \dotsb \otimes \mathcal{L}_{{T_m}}^{\otimes k_m}).$$
\end{corollary}

Another crucial feature of the Snapper polynomials of matroids is their valuativity, which will allow us to reduce Theorem~\ref{thm:projection} to the case of realizable polymatroids.

\begin{proposition}\label{prop:eulervaluative}
Let $\textbf{a} = (a_1, \dotsc, a_m)$ and $\textbf{k} = (k_1, \dotsc, k_m)$ be sequence of integers, with $a_i \ge 0$. For a polymatroid $P$ with cage \textbf{a}, let $\M$ be its multisymmetric lift with distinguished partition $S_1 \sqcup \dotsb \sqcup S_m$ of its ground set.  
Then the function which assigns to a polymatroid $\P$ with cage \textbf{a} the quantity $\chi(\M, \mathcal{L}_{{S_1}}^{\otimes k_1} \otimes \dotsb \otimes \mathcal{L}_{S_m}^{\otimes k_m})$ is valuative. 
\end{proposition}

\begin{proof}
By \cite[Lemma 3.2]{ELPoly}, the function that sends a polymatroid of cage \textbf{a} to the class of its multisymmetric lift in the valuative group of matroids is valuative. By \cite[Lemma 6.4]{LLPP}, for fixed $S_i$ and $k_j$, the function that sends a matroid $\M$ to $\chi(\M, \mathcal{L}_{{S_1}}^{\otimes k_1} \otimes \dotsb \otimes \mathcal{L}_{S_m}^{\otimes k_m})$ is a valuative invariant of matroids. Putting these together implies the result. 
\end{proof}

\subsection{Multiplicity-free subvarieties and the proof of Theorem~\ref{thm:projection}}\label{ssec:proofprojection}
An integral subvariety $X$ of a product of projective spaces $\mathbb{P}^{a_1} \times \dotsb \times \mathbb{P}^{a_m}$ is said to be \emph{multiplicity-free} if the intersection number of any monomial in the hyperplane classes of the factors with the fundamental class of $X$ is either $0$ or $1$.
By \cite[Corollary 4.7]{BH}, the function $\rk_\P \colon  2^{[m]} \to \ZZ$ defined by
\[
\rk_\P(I) = \dim \big(\text{image of $X$ under the projection to $\prod_{i\in I} \PP^{a_i}$}\big)
\]
is a polymatroid $\P$, which we refer to as the polymatroid of $X$.
The $K$-class of the structure sheaf $[\mathcal{O}_{X}] \in K(\mathbb{P}^{a_1} \times \dotsb \times \mathbb{P}^{a_m})$ is then determined by the following theorem. 

\begin{proposition}\cite{BrionMultiplicity}\label{prop:brion}
There is a flat degeneration of $X$ to $Y_\P$ inside of $\mathbb{P}^{a_1} \times \dotsb \times \mathbb{P}^{a_m}$. In particular, $[\mathcal{O}_X] = [\mathcal{O}_{Y_\P}]$. 
\end{proposition}

The second statement follows from the first because the pairing 
$$K(\mathbb{P}^{a_1} \times \dotsb \times \mathbb{P}^{a_m}) \times K(\mathbb{P}^{a_1} \times \dotsb \times \mathbb{P}^{a_m}) \to \mathbb{Z} \text{ given by }(a, b) \mapsto \chi(\mathbb{P}^{a_1} \times \dotsb \times \mathbb{P}^{a_m}, ab)$$ is nondegenerate, and Euler characteristics are locally constant in proper flat families. This implies that the class of a subvariety in $K(\mathbb{P}^{a_1} \times \dotsb \times \mathbb{P}^{a_m})$ is locally constant in proper flat families.  

\medskip
We now state a formula for $[\mathcal{O}_{Y_\P}] \in K(\mathbb{P}^{a_1} \times \dotsb \times \mathbb{P}^{a_m})$. This formula originates in the work of Knutson, who studied the more general problem of calculating the $K$-class of a reduced union of Schubert varieties inside a homogeneous space. He showed that one can compute the $K$-class in terms of M\"{o}bius inversion on the poset of Schubert varieties. The special case of products of projective spaces was also proven in \cite[Theorem 7.12]{CCMM}.
For each tuple $\textbf{b} = (b_1, \dotsc, b_m)$ with $b_i \le a_i$, let $Y_{\textbf{b}}$ be a $\mathbb{P}^{b_1} \times \dotsb \times \mathbb{P}^{b_m}$ embedded linearly into $\mathbb{P}^{a_1} \times \dotsb \times \mathbb{P}^{a_m}$; the class $[\mathcal{O}_{Y_{\textbf{b}}}]$ does not depend on the choice of an embedding. The classes $\{[\mathcal{O}_{Y_{\textbf{b}}}]\}$ form a basis for $K(\mathbb{P}^{a_1} \times \dotsb \times \mathbb{P}^{a_m})$.

\begin{proposition}\label{prop:KnutsonSplitting} \cite{KnutsonSplitting}
Write $[\mathcal{O}_{Y_\P}] = \sum_{\textbf{b}} c_{\textbf{b}}[\mathcal{O}_{Y_{\textbf{b}}}]$. If $\sum b_i > \operatorname{rk}(\P)$, then $c_{\textbf{b}} = 0$. If $\sum b_i = \operatorname{rk}(\P)$, then 
$$c_{\textbf{b}} = \begin{cases} 1 & \text{if } \textbf b \in B(\P) \\ 0 & \text{otherwise}.\end{cases}$$
If $\sum b_i < \operatorname{rk}(\P)$, then $c_{\textbf{b}} = 1 -\sum_{\textbf{b}' > \textbf{b}} c_{\textbf{b}'}$.
\end{proposition}

\begin{proposition}\label{prop:structuresheafvaluative}
The function which assigns a polymatroid $\P$ with cage $(a_1, \dotsc, a_m)$ to $[\mathcal{O}_{Y_\P}] \in K(\mathbb{P}^{a_1} \times \dotsb \times \mathbb{P}^{a_m})$ is valuative. 
\end{proposition}

\begin{proof}
We show that, for every $\textbf{b} = (b_1, \dotsc, b_m)$ with $b_i \le a_i$, the function assigns a polymatroid $\P$ with cage $(a_1, \dotsc, a_m)$ to $c_{\textbf{b}}$ is valuative. This is clear if $\sum b_i \ge \operatorname{rk}(\P)$. The recursive formula $c_{\textbf{b}} = -\sum_{\textbf{b}' > \textbf{b}} c_{\textbf{b}'}$ then implies that it holds in general. 
\end{proof}

We first prove Theorem~\ref{thm:projection} in the case when $\P$ is realizable over $\mathbb{C}$. 

\begin{proposition}\label{prop:realizablecase}
Let $V_1, \dotsc, V_m$ be vector spaces over $\mathbb{C}$ of dimension $a_1, \dotsc, a_m$, and let $L \subseteq V_1 \oplus \dotsb \oplus V_m$ be a realization of a polymatroid $\P$ with cage $(a_1, \dotsc, a_m)$. Let $\M$ be the multisymmetric lift of $\P$, whose ground set is equipped with the distinguished partition $S_1 \sqcup \dotsb \sqcup S_m$. Let $W_L$ be the augmented wonderful variety of a realization of $\M$.  Then, for any $(k_1, \dotsc, k_m)$,
$$\chi(\M, \mathcal{L}_{{S_1}}^{\otimes k_1} \otimes \dotsb \otimes \mathcal{L}_{{S_m}}^{\otimes k_m}) = \chi(Y_\P, \mathcal{O}(k_1, \dotsc, k_m)).$$
\end{proposition}

\begin{proof}
Let $Y$ be the image of $W_L$ under the projection $p$ to $\mathbb{P}^{a_1} \times \dotsb \times \mathbb{P}^{a_m}$, or, equivalently, $Y$ is the closure of $L$ inside $\PP(V_1 \oplus \kk) \times \dotsb \times \PP(V_m \oplus \kk)$.
As $W_L$ is also a compactification of $L$, the map $W_L \to Y$ is birational. 
By \cite[Theorem 4.3]{BergetFink}, which is based on \cite[Theorem 5]{BrionOrbitClosure}, $Y$ has rational singularities.
As $W_L$ is smooth, we  have that $Rp_* \mathcal{O}_{W_L} = \mathcal{O}_Y$.
By the projection formula, we have that
$$Rp_* (\mathcal{L}_{{S_1}}^{\otimes k_1} \otimes \dotsb \otimes \mathcal{L}_{{S_m}}^{\otimes k_m}) =  \mathcal{O}(k_1, \dotsc, k_m).$$ 
Because $\chi(\M, -)$ agrees with $\chi(W_L, -)$, we have that
$$\chi(\M, \mathcal{L}_{{S_1}}^{\otimes k_1} \otimes \dotsb \otimes \mathcal{L}_{{S_m}}^{\otimes k_m}) = \chi(W_L, \mathcal{L}_{{S_1}}^{\otimes k_1} \otimes \dotsb \otimes \mathcal{L}_{{S_m}}^{\otimes k_m}) = \chi(Y, \mathcal{O}(k_1, \dotsc, k_m)).$$
To conclude, we note that $Y$ is an irreducible multiplicity-free subvariety by \cite{LiRatMap} or \cite[Corollary 1.4]{ELPoly}.
By Proposition~\ref{prop:brion}, $[\mathcal{O}_{Y}] = [\mathcal{O}_{Y_\P}] \in K(\mathbb{P}^{a_1} \times \dotsb \times \mathbb{P}^{a_m})$ as $Y$ is multiplicity-free, which implies the result. 
\end{proof}

\begin{proof}[Proof of Theorem~\ref{thm:projection}]
Fix $(k_1, \dotsc, k_m)$.
We may assume $\M$ is the multisymmetric lift of $\P$ by Corollary~\ref{cor:indepofM}.
When $\P$ is realizable over $\mathbb{C}$, the statement follows from Proposition~\ref{prop:realizablecase}.
By Proposition~\ref{prop:eulervaluative}, the function that assigns a polymatroid $\P$ with cage $(a_1, \dotsc, a_m)$ to $\chi(\M, \mathcal{L}_{{S_1}}^{\otimes k_1} \otimes \dotsb \otimes \mathcal{L}_{{S_m}}^{\otimes k_m})$ is valuative, and by Proposition~\ref{prop:structuresheafvaluative}, the same is true with the function that assigns $\P$ to $\chi(Y_\P, \mathcal{O}(k_1, \dotsc, k_m))$. 
Corollary~\ref{cor:valequal} thus implies the desired equality.
\end{proof}

\subsection{Non-augmented K-rings}\label{ssec:nonaug}

We now discuss the analogue of Theorem~\ref{thm:projection} for (non-augmented) $K$-rings of matroid. This section is not used until Section~\ref{ssec:LofP}. Let $\M$ be a loopless matroid.
The (non-augmented) \emph{Bergman fan} $\underline\Sigma_\M$ of a matroid $\M$ is the star fan of a particular ray in the augmented Bergman fan $\Sigma_\M$; see \cite[Definition 5.12]{EHL} for details.
In other words, its toric variety $X_{\underline{\Sigma}_{\M}}$ is a toric divisor on $X_{\Sigma_{\M}}$.
We define the \emph{(non-augmented) $K$-ring of $\M$}, denoted $\underline{K}(\M)$, to be the $K$-ring of $X_{\underline{\Sigma}_{\M}}$. As $X_{\underline{\Sigma}_{\M}}$ is a divisor on $X_{\Sigma_{\M}}$, there is a restriction map $K(\M) \to \underline{K}(\M)$. The restriction of $[\mathcal{L}_S]$ is denoted $[\underline{\mathcal{L}}_S]$.

\smallskip
The facts about the augmented $K$-ring (Proposition~\ref{prop:augKfeatures}) have analogues for the non-augmented $K$-ring $\underline{K}(\M)$ \cite{LLPP}.  More precisely, we have:
\begin{enumerate}[label = (\roman*)]
\item $\underline K(\M)$ is equipped with an ``Euler characteristic map'' $\underline\chi(\M,-) \colon \underline K(\M) \to \ZZ$.
\item $\underline{K}(\M)$ is generated as a ring by the restrictions $[\underline{\mathcal L}_S]$ of the classes $[\mathcal{L}_S]$.
\item When $\M$ has a realization $L\subseteq \kk^{\EE}$, let $\underline W_L$ be the \emph{wonderful variety} \cite{dCP95} defined as
\[
\underline W_L = \text{the closure of the image of $\PP L$ in } \prod_{\emptyset \subsetneq S \subseteq E} \PP(\kk^S)
\]
where $\PP L \to \PP(\kk^S)$ is the projectivization of the projection $L \to \kk^S$, and let $\underline{\mathcal L}_S$ be the pullback of $\mathcal O(1)$ from $\PP(\kk^S)$.  Then, identifying the $[\underline{\mathcal L}_S]$ in $\underline K(\M)$ and $K(\underline W_L)$ gives an isomorphism $\underline K(\M) \simeq K(\underline W_L)$ such that  $\underline\chi(\M,-) = \chi(\underline W_L, -)$.
\end{enumerate}

\medskip
We also have a formula for the Euler characteristic map $\underline{\chi}(\M, -) \colon \underline{K}(\M) \to \mathbb{Z}$ analogous to Proposition~\ref{prop:KHallRado}.
We say that a sequence $(S_1, \dotsc, S_m)$ of nonempty subsets of $E$ satisfies the \emph{dragon Hall--Rado} condition (with respect to $\M$) if
\[
\operatorname{rk}_{\M} \left( \bigcup_{i \in I} S_i \right) \ge 1 + |I|\quad\text{for every $\emptyset \not = I \subseteq [m]$}. 
\]
Moreover, we say that $\mathbf k = (k_1, \dotsc, k_m) \in \ZZ_{\geq 0}^m$ satisfies the \emph{dragon Hall--Rado} condition if the sequence $(S_1^{k_1}, \dotsc, S_m^{k_m})$, where $S_i^{k_i}$ denotes $S_i$ repeated $k_i$ times, satisfies the condition, or, equivalently if
\[
\operatorname{rk}_{\M} \left( \bigcup_{i \in I} S_i \right) \ge 1 +  \sum_{i \in I} k_i \quad\text{for every $\emptyset \not= I \subseteq [m]$}. 
\]

This defines a polymatroid on $\{S : \emptyset \subsetneq S \subseteq E\}$ whose bases are the $\textbf{k}$ satisfying the dragon-Hall--Rado condition with $\sum k_S = \operatorname{rk}(\M)  - 1$. We call this the \emph{dragon-Hall--Rado polymatroid}. 
The significance of the dragon-Hall--Rado condition for us comes from the following formula for $\underline{\chi}(\M, -)$. 

\begin{proposition}\label{prop:DHR} \cite[Collary 7.5]{LLPP}
We have that 
$$\underline{\chi}(\M, \bigotimes_{S} \underline{\mathcal{L}}_{S}^{\otimes t_S}) = \sum_{\textbf{k} \text{ satisfies dragon-Hall--Rado}} \textbf{t}^{(\textbf{k})}.$$
\end{proposition}

By comparing this with Proposition~\ref{prop:KHallRado} and using Theorem~\ref{thm:projection}, we obtain the following non-augmented analogue of the theorem.

\begin{corollary}\label{cor:wonderful}
Let $\M$ be a matroid with subsets $S_1, \dotsc, S_m$ of the ground set, and let $\P$ be the restriction of the dragon-Hall--Rado polymatroid to $S_1, \dotsc, S_m$. 
Then, for any line bundle $\mathcal{L}$ which is a tensor product of the $\underline{\mathcal{L}}_{{S_i}}$, we have 
$\underline{\chi}(\M, \mathcal{L}) = \chi(Y_\P, \mathcal{L})$. 
\end{corollary}

\section{Lorentzian property}\label{sec:3}

We briefly summarize Lorentzian polynomials and then prove Theorem~\ref{thm:lorentzian}.  Then, we explain the application to $K$-polynomials of multiplicity-free subvarieties. 

\subsection{Lorentzian Snapper polynomials}

Lorentzian polynomials were introduced in \cite{BH} as a generalization of stable polynomials in optimization theory and volume polynomials in algebraic geometry.

\begin{definition}
A homogeneous polynomial $f = \sum_{\mathbf k} c_{\mathbf k} \mathbf t^{\mathbf k} \in \RR[t_1, \dotsc, t_m]$ of degree $d$ with nonnegative coefficients is \emph{Lorentzian} if
\begin{enumerate}
\item the support $\{\mathbf k \in \ZZ^m_{\geq 0} : c_{\mathbf k} > 0\}$ of $f$ equals $B(\P) \cap \ZZ^m$ for some polymatroid $\P$ on $[m]$, and
\item any $(d-2)$-th partial derivative of $f$ is a quadratic form with at most one positive eigenvalue.
\end{enumerate}
\end{definition}

The \emph{normalization} $N(f)$ of a polynomial $f \in \RR[t_1, \dotsc, t_m]$ is the polynomial obtained by replacing each term $c_{\mathbf k}\mathbf t^{\mathbf k}$ in $f$ with $c_{\mathbf k}\frac{\mathbf t^{\mathbf k}}{\mathbf k!}$ where $\mathbf k! = k_1! \dotsm k_m!$.
We say that $f$ is \emph{denormalized Lorentzian} if $N(f)$ is Lorentzian.

\medskip
For an irreducible complete variety, the volume polynomial of a collection of nef divisors is Lorentzian \cite[Theorem 4.6]{BH}.
We now prove Theorem~\ref{thm:lorentzian}, which states that the Snapper polynomial of the line bundles $\{\mathcal{L}_S\}$ on a matroid is also Lorentzian after a minor transformation.

\medskip
As before, for $k \in \ZZ_{\geq 0}$, denote $t^{(k)} =  \binom{t + k -1}{k}$ and $t^{[k]} =\binom{t+k}{k}$, and for $\mathbf k \in \ZZ_{\geq 0}^m$, denote $\mathbf t^{(\mathbf k)} = t_1^{(k_1)} \cdots t_m^{(k_m)}$ and $\mathbf t^{[\mathbf k]} = t_1^{[k_m]} \cdots t_m^{[k_m]}$.
Let us recall the notation in Theorem~\ref{thm:lorentzian} that $\widetilde H(\mathbf t, t_0)$ is the homogenization of the polynomial $H(\mathbf t)$ defined by
\[
H(t_1, \dotsc, t_m) = \sum_{\mathbf k} a_{\mathbf k} \mathbf t^{\mathbf k} \quad\text{such that}\quad \chi\Big(\M, \mathcal L_{S_1}^{\otimes t_1} \otimes \dotsm \otimes \mathcal L_{S_m}^{\otimes t_m}\Big) = \sum_{\mathbf k} (-1)^{r-|\mathbf k|} a_{\mathbf k}  \mathbf t^{[\mathbf k]}
\]
for a matroid $\M$ on $E$ and $S_1, \dotsc, S_m\subseteq E$ whose restriction polymatroid has rank $r$.

\begin{proof}[Proof of Theorem~\ref{thm:lorentzian}]
By Corollary~\ref{cor:indepofM} and using the multisymmetric lift, we may assume that the matroid $\M$ on $E$ has rank $r$ also.
When one of $S_1, \dotsc, S_m$ has full rank, say $S_m$, the restriction polymatroid of $\M$ to $S_1, \dotsc, S_m$ is the same as if $S_m = E$.  So, we may set $S_m = E$.
The polynomial of interest is
\[
H(\mathbf t, t_E) = \sum_{\mathbf k, \ell} a_{\mathbf k, \ell} \mathbf t^{\mathbf k} t_E^{\ell} \quad\text{such that}\quad  \chi\Big(\M, \mathcal L_{S_1}^{\otimes t_1} \otimes \dotsm  \otimes \mathcal L_{S_{m-1}}^{\otimes t_{m-1}} \otimes \mathcal L_E^{ \otimes t_E}\Big) = \sum_{\mathbf k, \ell} (-1)^{r - |\mathbf k| - \ell} a_{\mathbf k, \ell} \mathbf t^{[\mathbf k]}t_E^{[\ell]}
\]
where the summation is over $(\mathbf k, \ell) \in \ZZ^{m-1}_{\geq 0}\times \ZZ_{\geq 0}$.
Let $\widetilde H(\mathbf t, t_E, t_0)$ be its homogenization.
We need show that $\widetilde H$ is denormalized Lorentzian.

For $(\mathbf k, \ell) \in \ZZ^{m-1}_{\geq 0}\times \ZZ_{\geq 0}$ with $|\mathbf{k}| \le r$, note that $(\mathbf k, \ell)$ satisfies the Hall--Rado condition, i.e., $(\mathbf k, \ell)\in I(\P)$, if and only if $(\mathbf k,\ell') \in I(\P)$ for all $\ell'$ such that $|\mathbf k|+\ell' \leq r$.
Thus, from Proposition~\ref{prop:KHallRado}, we compute
\begin{multline*}
\chi\Big(\M, \mathcal L_{S_1}^{\otimes t_1} \otimes \dotsm  \otimes \mathcal L_{S_{m-1}}^{\otimes t_{m-1}} \otimes \mathcal L_E^{ \otimes t_E}\Big)
= \sum_{(\mathbf k, \ell) \in I(\P)} \mathbf t^{(\mathbf k)}t_E^{(\ell)}\\
= \sum_{(\mathbf k, \ell) \in B(\P)} \textbf t^{(\mathbf k)} t_E^{[\ell]}
= \sum_{(\mathbf k, \ell) \in B(\P)} t_E^{[\ell]} \prod_{i = 1}^{m-1}\big(t_i^{[k_i]} - t_i^{[k_i - 1]} \big),
\end{multline*}
where we used the binomial identity $t^{[k]} = t^{(k)} + t^{(k-1)} + \dotsb + t^{(1)} + 1$ for the second equality, and the binomial identity $t^{(k)} = t^{[k]}- t^{[k-1]}$ for the third (observed in Notation~\ref{notate:()[]}).  That is, we find 
\begin{equation}\label{eq:lorentzian}
\begin{split}
&H(\mathbf t, t_E) = \sum_{(\mathbf k, \ell) \in B(\P)} t_E^{\ell} \prod_{i = 1}^{m-1}\big(t_i^{k_i} + t_i^{k_i - 1} \big),\\
&\text{so that}\quad \widetilde H(\mathbf t, t_E, t_0) = \sum_{(\mathbf k, \ell) \in B(\P)} t_E^{\ell} \prod_{i = 1}^{m-1}\big(t_i^{k_i} + t_0t_i^{k_i - 1} \big).
\end{split}
\end{equation}
Normalizing, we thus have
\[
N(\widetilde H)(\mathbf t, t_E, t_0) = \sum_{(\mathbf k, \ell) \in B(\P)} \frac{t_E^{\ell}}{\ell!} \cdot \left(\prod_{i = 1}^{m-1} \Big(1 + t_0 \frac{\partial}{\partial t_i}\Big)\right)\left( \frac{\mathbf t^{\mathbf k}}{\mathbf k !}\right).
\]
The exponential generating function over the lattice points of the base polytope of a polymatroid is Lorentzian \cite[Theorem 3.10]{BH}, and the operator $(1 + t_0 \frac{\partial}{\partial t_i})$ preserves Lorentzian polynomials \cite[Proposition 2.7]{BH}.  Hence, $N(\widetilde H)$ is Lorentzian, i.e., $\widetilde H$ is denormalized Lorentzian.
\end{proof}

\begin{remark}
In general, the Snapper polynomial of very ample divisors on an irreducible projective variety may not similarly give a denormalized Lorentzian polynomial.  For example, on $\PP^1 \times \PP^1$, consider the line bundles $\mathcal L_1 = \mathcal O(2,2)$ and $\mathcal L_2 = \mathcal O(1,1)$.  We have that
\[
\chi(X, \mathcal L_1^{\otimes t_1} \otimes \mathcal L_2^{\otimes t_2}) = (2t_1+t_2+1)^2 = 8 t_1^{[2]}+4 t_1^{[1]}t_2^{[1]}+2 t_2^{[2]} - 12 t_1^{[1]} - 5t_2^{[1]}+4.
\]
The normalization of the homogenization of this polynomial (after  removing the alternating signs and turning $t^{[k]}$ into $t^{k}$) is
\[
4{t}_{1}^{2}+4{t}_{1}{t}_{2}+{t}_{2}^{2} +12{t}_{0}{t}_{1}+5{t}_{0}{t}_{2} +2{t}_{0}^{2},
\]
whose Hessian matrix has signature $(+,+,-)$. See \cite[Section 5.2]{FerroniHigashitani} for a related example.  
\end{remark}

\subsection{Applications}\label{ssec:appLorent}
We now explain applications of Theorem~\ref{thm:lorentzian} to $K$-polynomials.
The connection stems from the following formal consequences of some binomial identities, whose proofs we omit.

\medskip
For a polynomial $\chi(t_1, \dots, t_m) \in \mathbb Q[t_1, \dotsc, t_m]$ where  each monomial has degree at most  $(a_1, \dotsc, a_m)$, we have
\[
\sum_{\mathbf k \geq 0} \chi(k_1, \dots, k_m) \mathbf{t}^{\mathbf k} = \frac{\mathcal K(\chi, \mathbf t)}{(1-t_1)^{a_1}\dotsm (1-t_k)^{a_m}}
\]
for some polynomial $\mathcal K(\chi; \mathbf t)$ of degree at most $(a_1, \dots, a_m)$.
The polynomial $\mathcal K(\chi; 1 - t_1, \dotsc, 1-t_m)$, denoted $\mathcal K(\chi; \mathbf 1 - \mathbf t)$, is equivalently described as
\[
\mathcal K(\chi, \mathbf 1 - \mathbf t)  = \sum_{\mathbf k} c_{\mathbf a - \mathbf k} \mathbf t^{\mathbf k} \qquad\text{where}\qquad \chi(t_1, \dotsc, t_m) = \sum_{\mathbf k} c_{\mathbf k} \mathbf t^{[\mathbf k]}.
\]

Now, suppose a subvariety $X\subseteq \PP^{a_1} \times \dotsb \times \PP^{a_m}$ has the property 
that $\chi(X, \mathcal O_X(\mathbf k)) = h^0(X, \mathcal O(\mathbf k))$ for all $\mathbf k\in \ZZ_{\geq 0}^m$.
For instance, an irreducible multiplicity-free subvariety satisfies this property \cite{BrionMultiplicity}.
In this case, with $\chi(t_1, \dotsc, t_m)$ as the polynomial  $\chi(X, \mathcal O(t_1, \dotsc, t_m))$, 
the polynomial $\mathcal K(\chi, \mathbf 1 - \mathbf t) = \sum_{\mathbf k} c_{\mathbf a - \mathbf k} \mathbf t^{\mathbf k}$ encodes the $K$-class $[\mathcal O_X] \in K(\prod_{i = 1}^m \PP^{a_i})$ of the structure sheaf of $X$, that is,
\[
[\mathcal O_X] = \sum_{\mathbf k} c_{\mathbf a - \mathbf k} [\mathcal O_{H_1}]^{k_1}\dotsm [\mathcal O_{H_m}]^{k_m}
\]
where $\mathcal O_{H_i}$ denotes the structure sheaf of $\PP^{a_1} \times \dotsb \times \PP^{a_{i-1}} \times H_i \times \PP^{a_{i+1}} \times \dotsb \times \PP^{a_m}$ for a hyperplane $H_i \subset \PP^{a_i}$.
Note that, in the notation of Proposition~\ref{prop:KnutsonSplitting}, we have $[\mathcal O_{Y_{\mathbf a - \mathbf k}}] = [\mathcal O_{H_1}]^{k_1}\dotsm [\mathcal O_{H_m}]^{k_m}$.

\medskip
The polynomial $\mathcal K(\chi, \mathbf 1 - \mathbf t)$ is sometimes called the \emph{twisted $K$-polynomial}. The authors of \cite{CCMM} showed that, for an irreducible multiplicity-free subvariety $X \subseteq \prod_i \PP^{a_i}$, its coefficients have alternating signs, i.e., $(-1)^{\dim X - |\mathbf k|}c_{\mathbf k} \geq 0$.
Over $\mathbb{C}$, Brion \cite{BrionPositive} showed, more generally, that an irreducible subvariety $X$ with rational singularities in a flag variety $G/P$ has the property that the expansion of $[\mathcal O_X] \in K(G/P)$ in terms of the structure sheaves of Schubert subvarieties has alternating signs.

\medskip
For a polymatroid $\P$ of rank $r$ on $[m]$, not necessarily arising from an irreducible multiplicity-free subvariety, the authors of \cite{CCMM} defined the polynomial $g_\P(t_0,t_1, \dotsc, t_m)$ by
\[
g_\P(t_0, t_1, \dotsc, t_m) = \sum_{\mathbf k \in I(\P)\cap \ZZ^m} c_{\mathbf k}(-t_0)^{r-|\mathbf k|}\textbf t^{\textbf k}
\]
where the $c_{\mathbf k}$ are defined by the recursive formula given in Proposition~\ref{prop:KnutsonSplitting}.

\begin{corollary}\label{cor:CCMMconj}
Let $\P$ be a polymatroid of rank $r$ on $[m]$ such that $\rk_\P(i) = r$ for some $i\in [m]$.
Then $g_\P$ is a denormalized Lorentzian polynomial.  In particular, its support is the set of lattice points of the base polytope of a polymatroid.\footnote{When the support of the homogenization $\widetilde f$ of an inhomogeneous polynomial $f$ satisfies this property, the authors of \cite{CCMM} say that the support of the polynomial $f$ is a \emph{generalized polymatroid}.}
\end{corollary}

\begin{proof}
Combining Proposition~\ref{prop:KnutsonSplitting} and Theorem~\ref{thm:projection}, we find that $g_\P$ is exactly the polynomial $\widetilde H$ in Theorem~\ref{thm:lorentzian}.  Hence, the corollary is a restatement of Theorem~\ref{thm:lorentzian}.
\end{proof}

When the condition ``$\rk_\P(i) = \rk_\P([m])$ for some $i\in [m]$'' in the corollary is removed, the statement about the support is \cite[Conjecture 7.18]{CCMM}, and whether $g_\P$ is denormalized Lorentzian is \cite[Question 7.21]{CCMM}.

\begin{remark}\label{rem:motivation}
A subset $A \subset \ZZ^m$ is \emph{M-convex} if it can be translated to be the set of lattice points of the base polytope of a polymatroid on $[m]$.
By noting that reflecting an M-convex subset by a coordinate hyperplane give an M-convex subset, 
one can observe the following.
Let $\chi(t_1, \dotsc, t_m)=\sum_{\mathbf k} c_{\mathbf k} \mathbf t^{[\mathbf k]}$ be a polynomial of total degree $r$, and let $\mathcal K(\chi,\mathbf 1 - \mathbf t) = \sum_{\mathbf k} c_{\mathbf a - \mathbf k} \mathbf t^{\mathbf k}$ as before.  One has implications (i)$\implies$(ii)$\iff$(iii) of the following statements:
\begin{enumerate}[label = (\roman*)]
\item The homogeneous polynomial $\sum_{\mathbf k} c_{\mathbf k} (-t_0)^{r-|\mathbf k|} \mathbf t^{\mathbf k}$ is denormalized Lorentzian.
\item The homogeneous polynomial $\sum_{\mathbf k} c_{\mathbf k} (-t_0)^{r-|\mathbf k|} \mathbf t^{\mathbf k}$ has M-convex support.
\item The homogenization of $\mathcal K(\chi,\mathbf 1 - \mathbf t)$ has M-convex support.
\end{enumerate}
The main result of \cite{CCMM} states that, when $\chi(t_1, \dotsc, t_m) = \chi(X, \mathcal O(t_1, \dotsc, t_m))$ for an irreducible multiplicity-free subvariety $X\subseteq \PP^{a_1} \times \dotsb \times \PP^{a_m}$,
the polynomial $\sum_{\mathbf k} c_{\mathbf k} (-t_0)^{r-|\mathbf k|} \mathbf t^{\mathbf k}$ has M-convex support.
Corollary~\ref{cor:CCMMconj} implies furthermore that if the projection $X \to \PP^{a_i}$ onto one of the factors is generically finite onto its image for some $i$, then the twisted $K$-polynomial is \emph{dually Lorentzian} in the sense of \cite{RSW}.
\end{remark}

\begin{remark}\label{rem:matroidcase}
Let $\P$ be the restriction polymatroid of a matroid $\M$ on $E$ to the collection $S_1, \dotsc, S_m$, and let $\widetilde H$ be the polynomial defined in Theorem~\ref{thm:lorentzian}.
Using results in \cite{EHL}, one can show that when $\P$ is a matroid, Theorem~\ref{thm:lorentzian} holds without the assumption ``$\rk_\P(i) = \rk_\P([m])$ for some $i\in [m]$.''
We sketch a proof here.  Let $r$ be the rank of $\P$.

Replacing $E$ by $S_1 \cup \dotsb \cup S_m$, we assume $\rk_\M(E) =  \rk_\P([m]) = r$.  Consider the polynomial
\[
H'(\mathbf t, t_E) = \sum_{\mathbf k, \ell} a_{\mathbf k, \ell} \mathbf t^{\mathbf k} t_E^{\ell} \quad\text{such that}\quad  \chi\Big(\M, \mathcal L_{S_1}^{\otimes t_1} \otimes \dotsm  \otimes \mathcal L_{S_m}^{\otimes t_m} \otimes \mathcal L_E^{ \otimes t_E}\Big) = \sum_{\mathbf k, \ell} (-1)^{r - |\mathbf k| - \ell} a_{\mathbf k, \ell} \mathbf t^{[\mathbf k]}t_E^{[\ell]}
\]
where the summation is over $(\mathbf k, \ell) \in \ZZ^m_{\geq 0}\times \ZZ_{\geq 0}$.
Let $\widetilde H'(\mathbf t, t_E, t_0)$ be its homogenization.
Because $0^{[\ell]} = 1$, and $\widetilde H$ and $\widetilde H'$ both have degree $r$, setting $- t_E = t_0$ in $\widetilde H'$ gives the originally desired $\widetilde H(\mathbf t, t_0)$.  Combining this with the formula \eqref{eq:lorentzian}, we find that
\[
\text{the coefficient of $\mathbf t^{\mathbf k}t_0^{r-|\mathbf k|}$ in $\widetilde H(\textbf t, t_0)$ is }\sum_{\substack{J \subseteq [m] \text{ such that}\\ \mathbf k + \be_J \in B(\P)}}  (-1)^{|J|}
\]
where $\be_J = \sum_{J\in J} \be_j \in \RR^m$ denotes the sum of the standard basis vectors of $J$.

Now, if $\P$ is a matroid $\mathrm{N}$ on $[m]$, the displayed equation implies that
\[
\widetilde H(\mathbf t, t_0) = \sum_{\substack{I \subseteq [m]\\ \be_I \in I(\mathrm{N})}} T_{\mathrm{N}/I}(0,1) \mathbf t^{\be_I} t_0^{r - |I|}
\]
where $T_{\mathrm{N}/I}$ is the Tutte polynomial of the contraction matroid $\mathrm{N}/I$.
The right-hand-side polynomial $ \sum_{I \in I(\mathrm{N})} T_{\mathrm{N}/I}(0,1) \mathbf t^{\be_I} t_0^{r - |I|}$ is obtained from a denormalized Lorentzian polynomial in variables $x, z, w, u_1, \dotsc, u_m$ provided in \cite[Theorem 1.4 and Remark 8.9]{EHL} via the following two steps.  One keeps only the terms exactly divisible by $w^{m-r}$, and then sets $x = 0, z = t_0, u_i = t_i$.  Both steps preserve denormalized Lorentzian polynomials, and hence $\widetilde H(\mathbf t, t_0)$ is denormalized Lorentzian when $\P$ is a matroid $\mathrm{N}$.
\end{remark}

\section{$h^*$-vectors for matroids}\label{sec:4}
In this section, we define and study $h^*$-vectors of line bundles in $\underline{K}(\M)$. Let $\M$ be a loopless matroid.  For a line bundle $\mathcal{L}$ on $X_{\underline{\Sigma}_{\mathrm{M}}}$, it follows from Proposition~\ref{prop:KHallRado} that the function $t \mapsto \underline{\chi}(\M, \mathcal{L}^{\otimes t})$ is a polynomial in $t$, which we call the \emph{Snapper polynomial} of $\mathcal{L}$ on $\M$. 

\begin{definition}
For a loopless matroid $\M$ on a ground set $E$ and a line bundle $\mathcal L$ in $\underline K(\M)$, we define its \emph{$h^*$-vector} $(h^*_0(\M,\mathcal L), \dotsc, h^*_{d}(\M,\mathcal L))$ by
\[
\sum_{k \geq 0}\underline\chi(\M, \mathcal L^{\otimes k}) q^k = \frac{h^*(\M,\mathcal L; q)}{(1-q)^{d+1}} \quad \text{ where } \quad
h^*(\M,\mathcal L; q) = \sum_{k = 0}^{d} h^*_k(\M,\mathcal L) q^k,
\]
and $d$ is the degree of the Snapper polynomial of $\mathcal{L}$.
\end{definition}

Theorem~\ref{thm:macaulay} states that the $h^*$-vector is a Macaulay vector when $\mathcal L = \bigotimes_{S\subseteq E} \underline{\mathcal L}_{S}^{\otimes k_S}$ with $k_S \geq 0$ for all $S$.
In this section, we prove this theorem. 

\smallskip
In Section~\ref{ssec:macaulay}, we review Macaulay vectors and their relation to Cohen--Macaulayness and cohomology vanishing.  In Section~\ref{ssec:proofmacaulay}, we use properties of $Y_\P$ to prove Theorem~\ref{thm:macaulay}.
A generalization of Theorem~\ref{thm:macaulay} is conjectured in Section~\ref{ssec:LofP}.
Results on the degree of Snapper polynomials, necessary for studying $h^*$-vectors, are given in Section~\ref{ssec:numericaldim}.

\subsection{Macaulay vectors}\label{ssec:macaulay}
Recall that the Hilbert function of a graded algebra over a field $\kk$ is the sequence of the $\kk$-dimensions of the graded pieces.
For the numerical properties we consider, we may extend scalars to an extension of $\kk$, so we may assume $\kk$ is infinite as needed.

\begin{definition}\label{def:Macaulayvector}
A sequence $(h_0, h_1, \dotsc, h_d)$ is a \emph{Macaulay vector} if $(h_0, h_1, \dotsc, h_d, 0, 0, \dotsc)$ is the Hilbert function of a graded artinian $\kk$-algebra $A^{\bullet}$ which is generated in degree $1$ and has $A^0 = \kk$. 
\end{definition}

Macaulay vectors are also called M-vectors and O-sequences.
Macaulay gave an explicit description of these vectors as follows \cite[Theorem 4.2.10]{BH93}.
Given positive integers $n$ and $d$, there is a unique expression
$$n = \binom{k_d}{d} + \binom{k_{d - 1}}{d - 1} + \dotsb + \binom{k_{\delta}}{\delta}, \quad k_d > k_{d-1} > \dotsb > k_{\delta} \ge 1.$$
Set $n^{\langle d \rangle} = \binom{k_d + 1}{d + 1} + \dotsb + \binom{k_{\delta} + 1}{\delta + 1}$. Then $(1, a_1, \dotsc, a_d)$ is a Macaulay vector if and only if $0 \le a_{t + 1} \le a_t^{ \langle t \rangle}$ for all $t \ge 1$. 

\smallskip
Macaulay vectors often appear in the following way.
Suppose $R^\bullet$ is a graded Cohen--Macaulay algebra of Krull dimension $d+1$ with $R^0 = \kk$.
If the quotient of $R^\bullet$ by the ideal generated by $R^1$ is artinian, then $R^\bullet$ admits a linear system of parameters \cite[Propositions 1.5.11 and 1.5.12]{BH93}.
In this case, the quotient by a linear system of parameters is a graded artinian algebra $A^\bullet$ with the property that
\[
\sum_{k \geq 0} (\dim_\kk R^k) q^k = \frac{\dim_\kk A^0 + (\dim_\kk A^1) q + \dotsb + (\dim_\kk A^d) q^d}{(1-q)^{d+1}}.
\]
See for instance \cite[Remark 4.1.11]{BH93}.
In particular, if $R^\bullet$ is generated in degree 1, then the numerator of its Hilbert series $\sum_{k \geq 0} (\dim_\kk R^k) q^k$ is a polynomial whose coefficients form a Macaulay vector.
For the proof of Theorem~\ref{thm:macaulay}, we record the following cohomological criterion for a section ring to be Cohen--Macaulay. 

\begin{proposition}\label{prop:ampleMacaulay}
Let $\mathcal{L}$ be an ample line bundle on a geometrically connected and geometrically reduced projective $\kk$-variety $X$ of dimension $d$. Suppose that $H^i(X, \mathcal{L}^{\otimes k}) = 0$ for all $i > 0$ when $k \ge 0$, and $H^i(X, \mathcal{L}^{\otimes k}) = 0$ for all  $i < d$ when $k < 0$. Then, the section ring
\[
R_{\mathcal{L}}^\bullet \coloneqq  \bigoplus_{k \ge 0} H^0(X, \mathcal{L}^{\otimes k})
\]
is a graded Cohen--Macaulay $\kk$-algebra with $R_{\mathcal L}^0 = \kk$.  If furthermore $R^\bullet_{\mathcal L}$ is generated in degree 1, then the sequence $(h_0, \dotsc, h_d)$ defined by
\begin{equation}\label{eq:definingh}
\sum_{k \ge 0} \chi(X, \mathcal{L}^{\otimes k}) q^k = \frac{h_0 + h_1q + \dotsb + h_dq^d}{(1 - q)^{d+1}}
\end{equation}
is a Macaulay vector.
\end{proposition}

\begin{proof}
The sequence $(h_0, \dotsc, h_d)$ is well-defined via \eqref{eq:definingh} because $\chi(X, \mathcal{L}^{\otimes k})$ is a polynomial in $k$ (see \cite[Section 4.3]{Sta12}).
Because $X$ is geometrically connected, geometrically reduced, and proper over $\operatorname{Spec} \kk$, we have $R_{\mathcal{L}}^0 =\kk$.
Because all of the higher cohomology vanishes, we have $\chi(X, \mathcal{L}^{\otimes k}) = \dim H^0(X, \mathcal{L}^{\otimes k})$ for $k \ge 0$. Therefore the second statement follows from the first by our discussion above about Macaulay vectors.

It remains to show that $R_{\mathcal{L}}^{\bullet}$ is a Cohen--Macaulay graded ring.  That is, we show that the local cohomology $H^i_{\mathfrak{m}}(R_{\mathcal{L}}^{\bullet}; R_{\mathcal{L}}^{\bullet})$ with respect to the irrelevant ideal $\mathfrak m$ of $R_{\mathcal{L}}^{\bullet}$ vanishes for $i < d + 1$. The vanishing when $i=0, 1$ is automatic since $R_{\mathcal L}^{\bullet}$ is the section ring of $\mathcal O(1)$ on $X = \operatorname{Proj} R_{\mathcal L}^{\bullet}$.
For $i\geq 2$, we have $H^i_{\mathfrak{m}}(R_{\mathcal{L}}^{\bullet}; R_{\mathcal{L}}^{\bullet}) = \bigoplus_{k \in \mathbb{Z}} H^{i-1}(\operatorname{Proj} R_{\mathcal{L}}^{\bullet}, \mathcal{L}^{\otimes k})$ by \cite[Theorem 20.4.4]{BrodmannSharp}.
As $X = \operatorname{Proj} R_{\mathcal{L}}^{\bullet}$, the sheaf cohomology vanishing hypothesis gives desired vanishing of local cohomology.
\end{proof}

\subsection{Properties of $Y_\P$ and Theorem~\ref{thm:macaulay}}\label{ssec:proofmacaulay}

Let $\P$ be a polymatroid with cage $(a_1, \dotsc, a_m)$, and let $Y_\P\subseteq \PP^{a_1} \times \dotsm \times \PP^{a_m}$ be the subvariety defined in the introduction.
We note that $Y_\P$ is Cohen--Macaulay and compatibly Frobenius split, and we use these properties of prove Theorem~\ref{thm:macaulay}.

\begin{proposition}\label{prop:CM}
The variety $Y_\P$ is Cohen--Macaulay.
\end{proposition}

\begin{proof}
When there is a multiplicity-free subvariety $X \subseteq \PP^{a_1} \times \dotsm \times \PP^{a_m}$ whose polymatroid is $\P$, the Cohen--Macaulayness of $Y_\P$ is proven in \cite{BrionMultiplicity} via a geometric argument.  For arbitrary $\P$, the proposition is \cite[Proof of Theorem 5.6]{CCRC}, which was obtained by using properties of ``polymatroid ideals'' in \cite[Chapter 12.6]{HerzogHibi}.
\end{proof}

Note that $Y_{\P}$ is defined over $\operatorname{Spec}\mathbb{Z}$, with an embedding in a product of projective spaces over $\operatorname{Spec}\ZZ$. 
Viewing the product of projective spaces as a homogeneous space, $Y_{\P}$ is a reduced union of Schubert varieties, and hence it is a  compatibly Frobenius split subvariety of the product of projective spaces when base changed to any positive characteristic field $\kk$ \cite[Proposition 1.2.1, Theorem 2.3.10]{BrionKumar}. Together with Proposition~\ref{prop:CM}, this gives the following strong cohomology vanishing results for $Y_{\P}$.

\begin{proposition}\label{prop:Fsplit}  
Let $\mathcal{L}$ be the restriction of a very ample line bundle from the product of projective spaces to $Y_{\P}$. Then, we have $H^i(Y_{\P}, \mathcal{L}^{\otimes k}) = 0$ for all $i > 0$ when $k \ge 0$, and $H^i(Y_\P, \mathcal L^{\otimes k}) = 0$ for all $i < \operatorname{rk}(P)$ when $k < 0$.  Moreover, $Y_{\P}$ is geometrically reduced and geometrically connected, and the section ring $R_{\mathcal{L}}^{\bullet} = \bigoplus_{k \ge 0} H^0(Y_{\P}, \mathcal{L}^{\otimes k})$ is generated in degree $1$.
\end{proposition}

\begin{proof}
The cohomology vanishing follows from \cite[Theorem 1.2.8(ii), Theorem 1.2.9]{BrionKumar} because $Y_{\P}$ is Cohen--Macaulay. By \cite[Theorem 1.2.8(ii)]{BrionKumar}, $Y_{\P}$ is projectively normal in the embedding given by $\mathcal{L}$, so $R_{\mathcal{L}}^{\bullet}$ is generated in degree $1$. It remains to check that $Y_{\P}$ is geometrically reduced and geometrically connected. That it is geometrically reduced is obvious; it is geometrically connected because each component of $Y_{\P}$ contains the point $[1, 0, \dotsc, 0] \times [1, 0, \dotsc, 0] \times \dotsb \times [1, 0, \dotsc, 0]$. 
\end{proof}

\begin{proof}[Proof of Theorem~\ref{thm:macaulay}]
Let $\mathcal{L} = \bigotimes_{i = 1}^m \underline{\mathcal{L}}_{S_i}^{\otimes k_i}$ for some subsets $S_1, \dotsc, S_m$ of the ground set $E$ of the matroid $\M$ and integers $k_i >0$.
Let $\P$ be the restriction of the dragon-Hall--Rado polymatroid to the subsets $S_1, \dotsc, S_m$.
By Corollary~\ref{cor:wonderful}, we have that $\underline{\chi}(\M, \mathcal{L}^{\otimes \ell}) = \chi(Y_{\P},\mathcal{O}(k_1, \dotsc, k_m)^{\otimes \ell})$. Note that $\mathcal{O}(k_1, \dotsc, k_m)$ is the restriction of an ample divisor from the product of projective spaces to $Y_{\P}$. By Proposition~\ref{prop:Fsplit}, we have that $Y_\P$ and $\mathcal{O}(k_1, \dotsc, k_m)$ satisfy the conditions of Proposition~\ref{prop:ampleMacaulay}, including the generation of $\bigoplus_{k \ge 0} H^0(Y_\P, \mathcal{O}(k_1, \dotsc, k_m)^{\otimes k})$ in degree 1.  Hence, we conclude that $h^*(\M,\mathcal L)$ is a Macaulay vector.
\end{proof}

\subsection{Line bundles from polymatroids}\label{ssec:LofP}

We conjecture a generalization of Theorem~\ref{thm:macaulay}.  In Section~\ref{sec:applications}, we explain how the conjecture contains a question of Speyer \cite{Speyerg} as a special case, and how Theorem~\ref{thm:macaulay} answers the question for a new family of cases.
To do so, it is convenient to phrase the line bundles in $\underline K(\M)$ in terms of divisors in the non-augmented Chow ring $\underline A^\bullet(\M)$.

\begin{definition} \cite{FY}
Let $\M$ be a loopless matroid on a ground set $E$.  The \emph{non-augmented Chow ring} of $\M$ is the graded ring
\[
\underline A^\bullet(\M) = \frac{\ZZ[z_F : F \text{ a nonempty flat of }\M]}{\langle z_F z_{F'} : F \subseteq F' \text{ and } F \supseteq F' \rangle + \langle \sum_{F\ni i} z_F : i \in E\rangle}.
\]
An element of $\underline A^1(\M)$ is called a \emph{divisor class} on $\M$.
Equivalently, $\underline{A}^{\bullet}(\M)$ is the Chow ring of the toric variety $X_{\underline\Sigma_\M}$ of the non-augmented Bergman fan $\underline\Sigma_\M$ of $\M$.
\end{definition}

For a nonempty subset $S\subseteq E$, define an element $\underline h_S \in \underline A^1(\M)$ by
\[
\underline h_S = \sum_{F\supseteq S} - z_F.
\]
Because $\underline{K}(\M) = K(X_{\underline{\Sigma}_{\M}})$ and $\underline{A}^{\bullet}(\M) = A^{\bullet}(X_{\underline{\Sigma}_{\M}})$, one has a \emph{Chern class map} $c \colon \underline K(\M) \to \underline A^\bullet(\M)^{\times}$, a homomorphism from the additive group of $\underline{K}(\M)$ to the units in $\underline{A}^{\bullet}(\M)$, see \cite[Section 15.3]{FultonIntersection}. It has the characterizing property that 
\[
c(\underline{\mathcal L}_S) = 1+ \underline h_S \quad\text{for all nonempty $S\subseteq E$}.
\]
That is, we have $c_1(\underline{\mathcal L}_S) = \underline h_S$.  More generally, for a polymatroid $\P$ on $E$, let us define the line bundle $\underline{\mathcal L}_{B(\P)}$ in $\underline K(\M)$ via the property
\[
c_1(\underline{\mathcal L}_{B(\P)}) = \sum_F \big(\rk_\P(E\setminus F)  - \rk_\P(E) \big) z_F.
\]
One recovers $\underline{\mathcal L}_S$ via the polymatroid whose rank function is $\rk(I) = 1$ if $I \cap S \neq \emptyset$ and $0$ otherwise.

\begin{remark}\label{rem:chernclass}
These constructions have the following geometric origin.
When $\M$ has a realization $L\subseteq \kk^E$, the Chow ring $\underline A^\bullet(\M)$ coincides with the Chow ring of the wonderful variety $\underline W_L$ \cite{dCP95}, and the Chern class map $\underline K(\M) \to \underline A^\bullet(\M)$ coincides with the Chern class map $K(\underline W_L) \to A^\bullet(\underline W_L)$.

When furthermore $\M$ is the Boolean matroid $\mathrm{U}_{|E|,E}$, whose realization is $L = \kk^E$, the wonderful variety $\underline W_L$ is a toric variety $\underline X_E$ known as the \emph{permutohedral variety}.  In this case, under the standard correspondence between nef divisor classes on toric varieties and polytopes \cite[Chapter 6]{CLS}, the divisor class $c_1(\underline{\mathcal L}_{B(\P)})$ corresponds to the base polytope $B(\P)$.  Moreover, every nef divisor class is equal to $c_1(\underline{\mathcal L}_{B(\P)})$ for some polymatroid $\P$.  See \cite[Section 2.7]{BEST} and references therein.
\end{remark}

We conjecture the following positivity for $h^*$-vectors of line bundles from polymatroids.

\begin{conjecture}\label{conj:Mvector}
Let $\M$ be a loopless matroid on $E$, and let $\P$ be a polymatroid on $E$. Then, the $h^*$-vector  $h^*(\M, \underline{\mathcal{L}}_{B(\P)})$ is a Macaulay vector and is in particular nonnegative.
\end{conjecture}

Theorem~\ref{thm:macaulay} states that Conjecture~\ref{conj:Mvector} holds when $c_1(\underline{\mathcal L}_{B(\P)})$ is a nonnegative linear combination of the $\underline h_S$.
Several other cases in which Conjecture~\ref{conj:Mvector} holds are discussed in Section~\ref{ssec:future}.

\subsection{Degree of Snapper polynomials and numerical dimension}\label{ssec:numericaldim}
To study $h^*$-vectors arising from line bundles $\underline{\mathcal L}_{B(\P)}$ in Conjecture~\ref{conj:Mvector}, one needs some tools to understand the degree of the Snapper polynomial, since the degree is essential in the definition of $h^*(\M, \underline{\mathcal L}_{B(\P)})$.
One such tool is given in terms of the following.

\begin{definition}
The \emph{numerical dimension} of a line bundle $\mathcal{L}$ in $\underline{K}(\M)$ is the largest nonnegative integer $k$ such that $c_1(\mathcal{L})^k \not= 0$ in $\underline A^\bullet(\M)$.
\end{definition}

Our main result for numerical dimensions is the following.

\begin{theorem}\label{thm:numericaldim}
Let $\M$ be a loopless matroid or rank $r$ on a ground set $E$.
\begin{enumerate}[label = (\arabic*)]
\item \label{numericaldim:bound}For $\mathcal L$ a line bundle in $\underline K(\M)$, the degree of the Snapper polynomial $\underline\chi(\M, \mathcal L^{\otimes t})$ is at most the numerical dimension of $c_1(\mathcal L)$.  Moreover, the degree equals $r-1$ if and only if the numerical dimension is $r-1$.
\item \label{numericaldim:fulldim}For $\P$ a polymatroid on $E$ such that the base polytope $B(\P)$ is full dimensional (i.e., $(|E|-1)$-dimensional), then its numerical dimension is $r-1$, so the degree of the Snapper polynomial of $\underline{\mathcal L}_{B(\P)}$ is $r-1$.
\end{enumerate}
\end{theorem}

To prove Theorem~\ref{thm:numericaldim}\ref{numericaldim:bound}, we develop a version of the Hirzebruch--Riemann--Roch theorem for $K$ and Chow rings of matroids.
For this, we recall that the Chow ring $\underline A^\bullet(\M)$ is equipped with a degree map $\underline{\deg}_\M \colon \underline A^{r-1} \overset\sim\to \ZZ$
that satisfies Poincar\'e duality.   See \cite[Section 6]{AHK18} for details.

\begin{proposition}
There is a ring homomorphism ${\operatorname{ch}} \colon \underline{K}(\M) \to \underline{A}(\M)_{\mathbb{Q}}$ which induces an isomorphism $\underline{K}(\M)_{\mathbb{Q}} \to \underline{A}^{\bullet}(\M)_{\mathbb{Q}}$ defined by
\[
{\operatorname{ch}}([\mathcal{L}]) = \exp(c_1(\mathcal L)) = 1 + c_1(\mathcal{L}) + c_1(\mathcal{L})^2/2! + \dotsb.
\]
There is a class $\underline{\operatorname{Todd}}_{\M} \in \underline{A}^{\bullet}(\M)_{\mathbb{Q}}$ such that, for any $\xi \in \underline{K}(\M)_{\mathbb{Q}}$,
\[
\underline{\chi}(\M, \xi) = \underline{\deg}_{\M} \big({\operatorname{ch}}(\xi) \cdot \underline{\operatorname{Todd}}_{\M}\big).
\]
Moreover, the degree 0 part of $\underline{\operatorname{Todd}}_\M$ is 1.
\end{proposition}

\begin{proof}
We first recall $\underline{K}(\M) = K(X_{{\underline{\Sigma}_{\M}}})$ and $\underline{A}^{\bullet}(\M) = A^{\bullet}(X_{\underline{\Sigma}_{\M}})$, i.e., the $K$ and Chow rings of the toric variety $X_{\underline{\Sigma}_\M}$ (respectively).
Hence, that the Chern character map $\operatorname{ch}$ is well-defined and is an isomorphism after tensoring with $\mathbb{Q}$ is a general fact about algebraic varieties \cite[Example 15.2.16]{FultonIntersection}. 
Because $\underline{K}(\M)$ is generated by classes of line bundles \cite[Theorem 5.2]{LLPP}, the formula ${\operatorname{ch}}([\mathcal{L}]) = \exp(c_1(\mathcal{L}))$ determines $\operatorname{ch}$. 
By \cite[Theorem 6.19]{AHK18}, the pairing $\underline{A}^{\bullet}(\M)_{\mathbb{Q}} \otimes \underline{A}^{\bullet}(\M)_{\mathbb{Q}} \to \mathbb{Q}$ given by $(x, y) \mapsto \underline{\deg}_{\M}(x \cdot y)$ is a perfect pairing. Therefore there is some class $\underline{\operatorname{Todd}}_{\M} \in \underline{A}^{\bullet}(\M)_{\mathbb{Q}}$ such that the linear functional $x \mapsto \underline{\chi}(\M, \operatorname{ch}^{-1}(x))$ on $\underline{A}^{\bullet}(\M)_{\mathbb{Q}}$ is given by $x \mapsto \underline{\deg}_{\M}(x \cdot \underline{\operatorname{Todd}}_{\M})$.  Lastly, the degree 0 part of $\underline{\operatorname{Todd}}_\M$, which is some number in $\QQ$, must be 1 because Proposition~\ref{prop:DHR} implies that the leading term of the polynomial $\underline\chi(\M,\underline{\mathcal L}_E^{\otimes t})$ is $t^{r-1}/(r-1)!$, whereas
\begin{align*}
\underline{\deg}_\M\big(c_1(\underline{\mathcal L}_E)^{r-1}\big) = \underline{\deg}_\M\big((-z_E)^{r-1}\big) 
&= \underline{\deg}_\M\Big( \big(\sum_{i \in F \subsetneq E} z_F\big)^{r-1}\Big) \quad\text{for any fixed $i\in E$}\\
&= 1 \quad\text{by \cite[Proposition 5.8]{AHK18}.}\qedhere
\end{align*}
\end{proof}

\begin{proof}[Proof of Theorem~\ref{thm:numericaldim}]
Let $\mathcal{L}$ be a line bundle of numerical dimension $d$. Because $c_1(\mathcal{L}^{\otimes t}) = tc_1(\mathcal{L})$, we have that
$$\underline{\chi}(\M, \mathcal{L}^{\otimes t}) = \underline{\deg}_{\M}((1 + tc_1(\mathcal{L}) + t^2c_1(\mathcal{L})^2/2! + \dotsb) \cdot \underline{\operatorname{Todd}}_{\M}).$$
Since $c_1(\mathcal{L})^{d+1} = 0$, we see that the right-hand side is a polynomial in $t$ whose leading term is $t^\ell \underline{\deg}_{\M}(c_1(\mathcal{L})^\ell \cdot \underline{\operatorname{Todd}}_{\M})/\ell!$  for the largest $0\leq \ell\leq d$ such that $ \underline{\deg}_{\M}(c_1(\mathcal{L})^\ell \cdot \underline{\operatorname{Todd}}_{\M}) \neq 0$.
Moreover, because the degree 0 part of $\underline{\operatorname{Todd}}_\M$ is 1, we have
$$\underline{\chi}(\M, \mathcal{L}^{\otimes t}) = \underline{\deg}_{\M}(c_1(\mathcal{L})^{r-1}) \frac{t^{r-1}}{(r-1)!} + O(t^{r-2}).$$ 
Thus, $\mathcal{L}$ has numerical dimension $r-1$ if and only if the Snapper polynomial has degree $r-1$.  We have proven the first statement \ref{numericaldim:bound}.

For second statement \ref{numericaldim:fulldim}, we only need show that the numerical dimension of $\underline{\mathcal L}_{B(\P)}$ is $r-1$ if $B(\P)$ is full dimensional.
When $B(\P)$ is full dimensional, the line bundle $\underline{\mathcal L}_{B(\P)}$ in $K(\mathrm{U}_{|E|,E})$ of the boolean matroid corresponds to a nef and big line bundle on the projective toric variety $\underline X_E$ (see Remark~\ref{rem:chernclass}).
By \cite[Corollary 2.2.7]{Lazarsfeld1}, we can write the first Chern class as the sum of an ample class and an effective divisor class (inside $A^{\bullet}(\underline{X}_E) \otimes \mathbb{Q}$). Restricting this to $\underline{A}^{\bullet}(\M)$, we get that $c_1(\underline{\mathcal{L}}_{B(\P)}) = A + E$, where $A$ is the restriction of an ample class from $\underline{X}_E$ and $E$ is the restriction of an effective class.

We now prove by induction on $k$ that $\underline{\deg}_{\M}(c_1(\underline{\mathcal{L}}_{B(\P)})^k A^{r - 1 - k}) > 0$, using Proposition~\ref{prop:combample} stated below.
The case $k = 0$ is Proposition~\ref{prop:combample}(1).
For $k > 0$, Proposition~\ref{prop:combample}(2) gives that
\begin{equation*}\begin{split}
\underline{\deg}_{\M}(c_1(\underline{\mathcal{L}}_{B(\P)})^k A^{r - 1 - k}) &= \underline{\deg}_{\M} (c_1(\underline{\mathcal{L}}_{B(\P)})^{k-1} A^{r - k}) + \underline{\deg}_{\M} (c_1(\underline{\mathcal{L}}_{B(\P)})^{k-1} E A^{r - 1 - k}) \\  
&\ge \underline{\deg}_{\M} (c_1(\underline{\mathcal{L}}_{B(\P)})^{k-1} A^{r - k}),
\end{split}\end{equation*}
which is positive by induction. 
\end{proof}

\begin{proposition}\label{prop:combample}
Let $\M$ be a loopless matroid of rank $r$. 
\begin{enumerate}
\item Let $A \in \underline{A}^1(\M)$ be the restriction of an ample class from $\underline{X}_E$. Then $\underline{\deg}_{\M}(A^{r-1}) > 0$. 
\item Let $\P_1, \dotsc, \P_{r-2}$ be polymatroids. Then, for any class $E \in \underline{A}^1(\M)$ which is a restriction of an effective divisor class on $\underline{X}_E$, $\underline{\deg}_{\M}(c_1(\underline{\mathcal{L}}_{\P_1}) \cdot \dotsb c_1(\underline{\mathcal{L}}_{\P_{r-2}}) \cdot E) \ge 0$. 
\end{enumerate}
\end{proposition}

One can deduce the proposition as a general statement about combinatorially nef divisors on a fan with nonnegative Minkowski weights.
To avoid developing such notions here, we indicate a proof in terms of the Hodge--Riemann relations for $\underline A(\M)$ proven in \cite{AHK18}.

\begin{proof}

The first statement is the Hodge--Riemann relations in degree 0 for $\underline A(\M)$ \cite[Theorem 1.4]{AHK18}. The second statement is a consequence of the mixed Hodge--Riemann relations in degree 0 \cite[Theorem 8.9]{AHK18}, when one notes that the divisors $c_1(\underline{\mathcal L}_{B(\P_i)})$ are nef (thus, a limit of ample classes), and that the star of a ray in the Bergman fan of the matroid is a product of Bergman fans of matroids \cite[Proposition 3.5]{AHK18}.
\end{proof}

\section{Applications, Examples, and Problems}\label{sec:applications}

In Section~\ref{ssec:gpoly}, we study a question of Speyer \cite{Speyerg} as an application of results developed in the previous section.
It is for this application that we have focused on the non-augmented setting, although analogous statements for the augmented setting also hold.
Examples for Conjecture~\ref{conj:Mvector} and some further general properties of $h^*$-vectors of matroids are given in Section~\ref{ssec:future}, along with future directions.

\subsection{Application to Speyer's $g$-polynomial}\label{ssec:gpoly}

In this section, we apply Theorem~\ref{thm:macaulay} to study Speyer's \emph{$g$-polynomial} of a matroid \cite{Speyerg}. 
For a loopless and coloopless matroid $\M$ of rank $r$ on $[n]$, the $g$-polynomial $g_{\M}(t)$ is a polynomial of degree at most $r$ defined in terms of the $K$-theory of the Grassmannian $Gr(r, n)$, first defined for matroids realizable over a field of characteristic $0$ in \cite{Speyerg} and then for all matroids in \cite{FinkSpeyer}.

\smallskip
An outstanding problem about the $g$-polynomial is to show that it always has nonnegative coefficients.
In \cite{Speyerg}, Speyer used the Kawamata--Viehweg vanishing theorem to show the nonnegativity for matroids realizable over a field of characteristic $0$.
This allowed him to bound the number of cells of each dimension in a subdivision of the hypersimplex into matroid polytopes when all of the cells correspond to matroids realizable in characteristic $0$.
Nonnegativity of $g_\M(t)$ for all matroids would bound the complexity of any such subdivision in general.
The nonnegativity was proved for all sparse paving matroids in \cite[Theorem 13.16]{FerroniSchroter}.
Using Theorem~\ref{thm:macaulay} we give a new infinite family of matroids for which the nonnegativity holds.

\medskip
We begin by explaining how the nonnegativity of the coefficients of the $g$-polynomial is a special case of Conjecture~\ref{conj:Mvector}.
For a loopless and coloopless matroid $\M$ of rank $r$, let $\omega(\M)$ be the $t^r$ coefficient of $g_{\M}(t)$.  In forthcoming work, Alex Fink, Kris Shaw, and David Speyer show the following result. 

\begin{proposition}\label{prop:omegatog}
Suppose that $\omega(\M) \ge 0$ for all connected matroids. Then all coefficients of $g_{\M}(t)$ are nonnegative for all loopless and coloopless matroids. 
\end{proposition}

The following result was communicated to the authors by Alex Fink, Kris Shaw, and David Speyer.

\begin{proposition}\label{prop:KOmega}
Let $\M$ be a matroid of rank $r$ with $c$ connected components, and denote by $B(\M^\perp)$ the base polytope of the dual matroid $\M^\perp$ of $\M$. Then, we have
$$\omega(\M) = (-1)^{r - c} \underline{\chi}(\M, \underline{\mathcal{L}}_{B(\M^{\perp})}^{-1}).$$
\end{proposition}

\begin{proof}
We sketch a proof using results from \cite{LLPP} and \cite{BEST}.
By \cite[Theorem 1.8]{LLPP}, we have
$$\underline{\chi}(\M, \underline{\mathcal{L}}_{B(\M^{\perp})}^{-1}) = \underline{\deg}_{\M}( \zeta_{\M}([\underline{\mathcal{L}}_{B(\M^{\perp})}^{-1}]) \cdot (1 + \underline{h}_E + \underline{h}_E^2 + \dotsb)),$$
where $\zeta_{\M}$ is defined in \cite{LLPP}. Computing in the equivariant Chow groups of the permutohedral variety $\underline{X}_E$ using \cite[Proposition 5.6]{LLPP} and \cite[Theorem 10.1]{BEST} (see \cite[Corollary 6.5]{EHL}), we have that $\zeta_{\M}([\underline{\mathcal{L}}_{B(\M^{\perp})}^{-1}])$ is the restriction to $\underline{A}^{\bullet}(\M)$ of the class denoted $c(\mathcal{Q}_{\M}^{\vee})$ in \cite{BEST}. Then the result follows from  \cite[Theorem 10.12]{BEST}.
\end{proof}

Now, recall the formal identity satisfied by the $h^*$-vector
\begin{equation}\label{eq:topcoeff}
h^*_d(\M, \mathcal{L}) = (-1)^{d} \underline{\chi}(\M, \mathcal{L}^{-1})
\end{equation}
(see for instance \cite[Section 4.3]{Sta12}).
Moreover, when $\M$ is connected, the polytope $B(\M^{\perp})$ is full dimensional, so Theorem~\ref{thm:numericaldim} implies that the degree $d$ of the Snapper polynomial is $r - 1$.
Therefore, the two preceding propositions show that the nonnegativity of the coefficients of $g_\M(t)$ is a special case of Conjecture~\ref{conj:Mvector} with $\P = \M^\perp$.

\medskip
We now make explicit how Theorem~\ref{thm:macaulay} proves the positivity of $\omega(\M)$ in some special cases.
The first step is to express $\underline{\mathcal L}_{B(\M^\perp)}$ as a Laurent monomial in the $[\underline{\mathcal L}_S]$, or, equivalently, write $c_1(\underline{\mathcal L}_{B(\M^\perp)})$ as a linear combination of the $\underline h_S$.
To do so, for a matroid $\M$, we recall its \emph{$\beta$-invariant} \cite{Crapo}, defined by two properties:
\begin{itemize}
\item $\beta(U_{0,1}) = 0$, $\beta(U_{1,1}) = 1$, and $\beta(\M)= 0$ if $\M$ is disconnected, and 
\item the recursive relation: for any $i$ which is not a loop or coloop of $\M$
, 
$$\beta(\M) = \beta(\M/i) + \beta(\M \setminus i).$$
\end{itemize}
Equivalently, the $\beta$-invariant is the coefficient of $x$ in the Tutte polynomial of $\M$.

\begin{proposition}
Let $\M$ be a matroid on $[n]$.  Then, the polytope $B(\M^\perp)$ satisfies
\[
c_1(\underline{\mathcal{L}}_{B(\M^{\perp})}) = \sum_{\substack{F \text{ connected flat} \\ \text{of } \operatorname{rk}_{\M}(F) \ge 2}} \sum_{\operatorname{cl}_{\M}(S) = F}(-1)^{|S| - \operatorname{rk}_{\M}(S) + 1} \beta(\M|_{S}) \underline{h}_F \in \underline A^\bullet(\M).
\]
\end{proposition}

\begin{proof}
Let $\underline{\Delta}_S$ be the simplex $\operatorname{Conv}(\{\textbf{e}_{i} : i \in S\})$.
Then \cite[Theorem 2.6]{ABD10} expressed $B(\M^\perp)$ as a signed Minkowski sum of these simplices as follows:
$$B(\M^{\perp}) = \sum_{S \subseteq [n], |S| \ge 2} (-1)^{|S| - \operatorname{rk}_{\M}(S) + 1} \beta(\M|_{S}) \underline{\Delta}_S + \sum_{i \text{ loop of }\M} \underline{\Delta}_i.$$
This gives an expression for $c_1(\underline{\mathcal{L}}_{B(\M^{\perp})})$ on the permutohedral toric variety as a sum of the simplicial generators $\underline h_S$. As $\underline h_S = \underline{h}_{\operatorname{cl}_{\M}(S)}$ in $\underline{A}^{\bullet}(\M)$ and $\underline{h}_i = 0$, we obtain the desired expression.
\end{proof}

\begin{theorem}\label{thm:simplicialomega}
Let $\M$ be a loopless and coloopless matroid of rank $r$ such that, for all connected flats $F$ of $\M$ of rank at least $2$, we have $\sum_{\operatorname{cl}_{\M}(S) = F}(-1)^{|S| - \operatorname{rk}_{\M}(S) + 1} \beta(\M|_{S}) \ge 0$. Then $\omega(\M) \ge 0$. 
\end{theorem}

\begin{proof}
First suppose that $\M$ is connected, so the polytope $B(\M^{\perp})$ is full dimensional.
By Theorem~\ref{thm:numericaldim}\ref{numericaldim:fulldim}, the degree of the Snapper polynomial of $\underline{\mathcal{L}}_{B(\M^{\perp})}$ is $r-1$ in this case.
By Theorem~\ref{thm:macaulay} along with \eqref{eq:topcoeff}, we thus have $\omega(\M) = (-1)^{r-1}\underline{\chi}(\M, \underline{\mathcal{L}}_{B(\M^{\perp})}^{-1}) = h^*_{r-1}(\M, \underline{\mathcal{L}}_{B(\M^{\perp})}) \ge 0$. 

Now suppose that $\M = \M_1 \oplus \dotsb \oplus \M_c$, with each $\M_i$ connected. The hypothesis implies that each $\underline{\mathcal{L}}_{B(\M_i^{\perp})}$ is simplicially positive for each $i$, and so $\omega(\M_i) \ge 0$. By \cite[Proposition 7.2]{FinkSpeyer}, 
$g_{\M}(t) = g_{\M_1}(t) \dotsb g_{\M_c}(t).$
Because the $t^i$ coefficient of $g_{\M}(t)$ vanishes for $i > \operatorname{rk}(\M)$, 
\begin{equation*}
\omega(\M) = \omega(\M_1)  \dotsb  \omega(\M_c) \ge 0. \qedhere
\end{equation*}
\end{proof}

In particular, Theorem~\ref{thm:simplicialomega} states that if $\underline{\mathcal{L}}_{B(\M^\perp)}$ is simplicially positive, then $\omega(\mathrm{M}) \ge 0$. While it appears that this is not often satisfied, Theorem~\ref{thm:simplicialomega} does show that $\omega(\M) \ge 0$ for many matroids. We give two examples. 

\begin{example}
For a nonempty subset $S\subseteq E$, let $\mathrm{H}_S$ be the corank 1 matroid on $E$ with $S$ as its unique circuit.  A \emph{co-transversal matroid} is a matroid $\M$ that arises as the matroid intersection $\M = \mathrm{H}_{S_1} \wedge \dotsb \wedge \mathrm{H}_{S_c}$ for some (not necessarily distinct) subsets $S_1, \dotsc, S_c$.  In this case, one verifies that $c_1(\underline{\mathcal{L}}_{B(\M^\perp)}) = \sum_{i = 1}^c \underline h_{S_i} \in \underline A^{\bullet}(\M)$.  In particular, $\underline{\mathcal{L}}_{B(\mathrm{M}^{\perp})}$ is simplicially positive, so Theorem~\ref{thm:simplicialomega} applies to all co-transversal matroids. 
\end{example}

Co-transversal matroids are realizable over an infinite field of arbitrary characteristic, so we could have used \cite[Proposition 3.3]{Speyerg} or Example~\ref{rem:BergetFink} below.  We now construct an infinite family of matroids to which Theorem~\ref{thm:simplicialomega} applies but which are not realizable over a field of characteristic $0$, as follows.
We will use the notion of \emph{principal extensions}, whose definition and properties can be found in \cite[\S7.2]{Oxl11}.

\begin{lemma}
Let $\M$ be a loopless matroid on $E$, and fix a nonempty flat $G$.
Denote by $\M' = \M +_G \star$ the principal extension of $\M$ by $G$.
Then, writing
\[
c_1(\underline{\mathcal{L}}_{B(\M^\perp)}) = \sum_{F \text{ a flat of $\M$}} c_F \underline h_F \in \underline A^{\bullet}(\M),
\]
we have that the expression for $c_1(\underline{\mathcal{L}}_{B((\M')^\perp)}) \in \underline A^{\bullet}(\M')$ is roughly ``obtained by increasing the coefficient of $c_G$ by 1,'' or precisely,
\[
c_1(\underline{\mathcal{L}}_{B((\M')^\perp)}) = \underline h_{G\cup \star} + \sum_{F \supseteq G} c_F \underline h_{F\cup \star} + \sum_{F\not\supseteq G} c_F \underline h_F. 
\]
\end{lemma}

\begin{proof}
We use the fact that $c_1(\underline{\mathcal{L}}_{B(\M^\perp)}) = \sum_F \operatorname{rk}_\M(F) z_F$ (see \cite[Section 2.7 and Remark III.1]{BEST}), so that the coefficients $c_F$ are defined by the property that $\sum_{F'\subseteq F} c_{F'} = \operatorname{rk}_\M(F)$ for all flats $F$ of $\M$.

Now, we recall that the set of flats of $\M'$ is partitioned into three categories \cite[Corollary 7.2.5]{Oxl11}:
\begin{enumerate}[label = (\roman*)]
\item $\{F : F \text{ a flat of $\M$ such that $F\not\supseteq G$}\}$, in which case $\rk_{\M'}(F) = \rk_\M(F)$,
\item $\{F \cup \star : F \text{ a flat of $\M$ such that $F\supseteq G$}\}$, in which case $\rk_{\M'}(F\cup \star) = \rk_\M(F)$, and
\item $\{ F \cup \star : F \text{ a flat of $\M$ such that $F\not\supseteq G$ and $F$ is not covered by an element in $[G,E]$}\}$, in which case $\rk_{\M'}(F\cup \star) = \rk_\M(F) + 1$.
\end{enumerate}
Thus, in $\underline A^{\bullet}(\M')$, since $\underline h_\star = 0$ so that $- z_{E\cup \star} = \sum_{\emptyset \subseteq F \subsetneq E} z_{F\cup \star}$, we have
\[
\underline h_{G\cup \star} = \sum_{\emptyset \subseteq F \subsetneq E} z_{F\cup \star} + \sum_{G \subseteq F \subsetneq E} - z_{F \cup \star} = \sum_{F\not\supseteq G} z_{F\cup \star}.
\]
The claimed expression for $c_1(\underline{\mathcal{L}}_{B((\M')^\perp)})$ in all three cases of flats now follows, as the above expression for $\underline h_{G\cup \star}$ contributes only to the case (iii) and not to cases (i) or (ii).  Explicitly, we have:
\begin{enumerate}[label = (\roman*)]
\item In this case, the coefficient of $z_F$ is $\sum_{F' \subseteq F} c_{F'} = \rk_\M(F) = \rk_{\M'}(F)$.
\item In this case, the coefficient of $z_{F\cup \star}$ is again $\sum_{F' \subseteq F} c_{F'} = \rk_\M(F) = \rk_{\M'}(F\cup \star)$.
\item In this case, the coefficient of $z_{F\cup \star}$ is $ 1+ \sum_{F'\subseteq F} c_{F'} = 1+\rk_\M(F) = \rk_{\M'}(F\cup \star)$.\qedhere
\end{enumerate}
\end{proof}

Given any matroid $\M$, repeatedly applying the lemma provides a method to construct a matroid $\widetilde\M$ for which Theorem~\ref{thm:simplicialomega} applies.
Moreover, a matroid is realizable over an infinite field $\kk$ if and only if its principal extensions are realizable over the same field $\kk$.  Thus, the matroid $\widetilde\M$ has the same realizability property as $\M$ over infinite fields.
In particular, the lemma produces infinite families of matroids that are not realizable or realizable only over certain positive characteristics for which $\underline{\mathcal{L}}_{B(\mathrm{M}^{\perp})}$ is simplicially positive, so Theorem~\ref{thm:simplicialomega} applies.

\subsection{Examples and problems}\label{ssec:future}

We present several cases in which Conjecture~\ref{conj:Mvector} holds.

\begin{example}
When $\M$ is the boolean matroid, the discussion in Remark~\ref{rem:chernclass} implies that $h^*(\M, \underline{\mathcal{L}}_{B(\P)})$ is the usual $h^*$-vector of the base polytope $B(\P)$, and hence is nonnegative.
Moreover, because base polytopes of polymatroids have the property that every lattice point in $kB(\P)$ is a sum of $k$ lattice points in $B(\P)$ (see \cite[Chapter 18.6, Theorem 3]{Wel76}), $h^*(\M, \underline{\mathcal{L}}_{B(\P)})$ is a Macaulay vector. 
\end{example}

\begin{example}
Let $\nabla = \operatorname{Conv}(\{(0, 1, \dotsc, 1), (1, 0, 1, \dotsc, 1), \dotsc, (1, 1, \dotsc, 1, 0)\})$, the base polytope of the uniform matroid of corank $1$, so $c_1(\underline{\mathcal{L}}_{\nabla}) \in \underline{A}^1(\M)$ is the class usually denoted $\beta$.  Then \cite[Lemma 8.5]{LLPP} implies that
$$\underline{\chi}(\M, \underline{\mathcal{L}}_{\nabla}^{\otimes t}) = \sum_{i}f_{r-1-i}(BC_>(\M)) \binom{t}{r-1-i},$$
where $f_j(BC_>(\M))$ is the number of $j$-dimensional faces of the reduced broken circuit complex of $\M$ under any ordering $>$. As $\binom{t}{r-1-i} = \sum_{j = 0}^i (-1)^j \binom{i}{j}\binom{t+i}{r-1}$, we may express the Snapper polynomial in terms of the $h$-vector of the reduced broken circuit complex:
$$\underline{\chi}(\M, \underline{\mathcal{L}}_{\nabla}^{\otimes t}) =\sum_i h_{r-1-i}(BC_>(\M))  \binom{t+i}{r-1}.$$
Comparing this with the definition of $h^*(\M, \underline{\mathcal{L}}_{\nabla})$, we see that $h_i(BC_>(\M)) = h_i^*(\M, \underline{\mathcal{L}}_{\nabla})$. By \cite{BjornerShellable}, the reduced broken circuit complex is shellable and hence Cohen--Macaulay, so its $h$-vector is a Macaulay vector. This argument is closely related to \cite{PS06}.
\end{example}

\begin{example}\label{rem:BergetFink}
Let $\M$ be a connected matroid that has a realization $L \subseteq \kk^E$ over a field of characteristic $0$. Then the base polytope of the dual matroid $B(\M^{\perp})$ is full dimensional. It follows from \cite[Theorem 5.1]{BergetFink} and \cite[Theorem 7.10]{BEST} that, for all $k > 0$, the restriction map $H^0(\underline{X}_E, \underline{\mathcal{L}}_{B(\M^{\perp})}^{\otimes k}) \to H^0(W_L, \underline{\mathcal{L}}_{B(\M^{\perp})}^{\otimes k})$ is surjective and that $H^i(W_L, \underline{\mathcal{L}}_{B(\M^{\perp})}^{\otimes k}) = 0$ for $i > 0$. Therefore, by \cite[Chapter 18.6, Theorem 3]{Wel76}, the ring 
$$R^{\bullet} \coloneqq \bigoplus_{k \ge 0} H^0(W_L, \underline{\mathcal{L}}_{B(\M^{\perp})}^{\otimes k})$$
is generated in degree $1$. This implies that $\operatorname{Proj} R^{\bullet}$ is the image of $W_L$ under the complete linear system of $\underline{\mathcal{L}}_{B(\M^{\perp})}$. This is called the \emph{space of visible contours} of $L$.  It is known that $\operatorname{Proj} R^{\bullet}$ has rational singularities \cite[Theorem 1.4 and 1.5]{TevelevCompactifications}. In particular,
$$H^i(W_L, \underline{\mathcal{L}}_{B(\M^{\perp})}^{\otimes k}) = H^i(\operatorname{Proj} R^{\bullet}, \mathcal{O}(k))$$
for all $i$ and $k$. Because $B(\M^{\perp})$ is full dimensional, the line bundle $\underline{\mathcal{L}}_{B(\M^{\perp})}$ is nef and big.  By the Kawamata--Viehweg vanishing theorem, $H^i(W_L, \underline{\mathcal{L}}_{B(\M^{\perp})}^{\otimes k}) = H^i(\operatorname{Proj} R^{\bullet}, \mathcal{O}(k)) = 0$ for $k < 0$ and $i < \dim W_L$. As $W_L$ is rational, $H^i(W_L, \mathcal{O}_{W_L}) = 0$ for $i > 0$. Then Proposition~\ref{prop:ampleMacaulay} implies that $R^{\bullet}$ is Cohen--Macaulay, and so $h^*(\M, \underline{\mathcal{L}}_{B(\M^{\perp})})$ is a Macaulay vector. 
\end{example}

Lastly, we discuss the valuativity of $h^*$-vectors of matroids and conjecture a monotonicity property for them. 
When $\P$ is full dimensional, Theorem~\ref{thm:numericaldim} implies that the degree of the Snapper polynomial of $\underline{\mathcal{L}}_{B(\P)}$ depends only on the rank of $\M$. From the formula for $h^*(\M, \underline{\mathcal{L}}_{B(\P)})$ and the valuativity of $\underline{\chi}(\M, \underline{\mathcal{L}}_{B(\P)}^{\otimes k})$ for fixed $\P$ and $k$ \cite[Lemma 6.4]{LLPP}, we obtain the following corollary. 

\begin{corollary}
Let $\P$ be a polymatroid such that $B(\P)$ is full dimensional. Then the function that assigns to a loopless matroid $\M$ the polynomial $h^*(\M, \underline{\mathcal{L}}_{B(\P)})$ is valuative. 
\end{corollary}

However, the degree of the Snapper polynomial of $\underline{\mathcal{L}}_{B(\P)}$ on $\M$ is not in general determined by the rank of $\M$. The numerical dimension of $\underline{\mathcal{L}}_{B(\P)}$ is also not determined by the rank of $\M$. 

\begin{question}
Let $\P$ be a polymatroid. What is the degree of the Snapper polynomial of $\underline{\mathcal{L}}_{B(\P)}$ on $\M$? Is it equal to the numerical dimension of $\underline{\mathcal{L}}_{B(\P)}$? Is $h^*(\M, \underline{\mathcal{L}}_{B(\P)})$ valuative?
\end{question}

We also conjecture the following monotonicity property for $h^*$-vectors, inspired by Stanley's monotonicity result for $h^*$-vectors of polytopes \cite{StanleyMonotonicity}, which implies the following conjecture when $\M$ is the boolean matroid. 

\begin{conjecture}\label{conj:monotonicity}
Let $\P_1, \P_2$ be polymatroids with $B(\P_1) \subseteq B(\P_2)$. Then for any loopless matroid $\M$, $h^*_i(\M, \underline{\mathcal{L}}_{B(\P_1)}) \le h^*_i(\M, \underline{\mathcal{L}}_{B(\P_2)})$ for all $i$.  
\end{conjecture}

If the degree of the Snapper polynomial of $\mathcal{L}$ is $\operatorname{rk}(\M) -1$,  then $\sum h_i^*(\M, \mathcal{L}) = \underline{\deg}_{\M}(c_1(\mathcal{L})^{r-1})$, so the following result gives evidence for Conjecture~\ref{conj:monotonicity}. 

\begin{proposition}
Let $\P_1, \P_2$ be polymatroids with $B(\P_1) \subseteq B(\P_2)$. Then 
$$\underline{\deg}_{\M}(c_1(\underline{\mathcal{L}}_{B(\P_1)})^{r-1}) \le \underline{\deg}_{\M}(c_1(\underline{\mathcal{L}}_{B(\P_2)})^{r-1}).$$
\end{proposition}

\begin{proof}
Because $B(\P_1) \subseteq B(\P_2)$, the difference of the divisor class in $A^1(\underline{X}_E)$ corresponding to $B(\P_2)$ with the divisor class corresponding to $B(\P_1)$ is an effective divisor class, see \cite[Section 2.7]{BEST}. Then,
\begin{equation*}\begin{split}
&c_1(\underline{\mathcal{L}}_{B(\P_2)})^{r-1} - c_1(\underline{\mathcal{L}}_{B(\P_1)})^{r-1} \\ 
&=  (c_1(\underline{\mathcal{L}}_{B(\P_2)}) - c_1(\underline{\mathcal{L}}_{B(\P_1)})) \cdot (c_1(\underline{\mathcal{L}}_{B(\P_2)})^{r-2} + c_1(\underline{\mathcal{L}}_{B(\P_2)})^{r-3}c_1(\underline{\mathcal{L}}_{B(\P_1)}) + \dotsb + c_1(\underline{\mathcal{L}}_{B(\P_1)})^{r-2}).
\end{split}\end{equation*}
By Proposition~\ref{prop:combample}, the degree of this class is nonnegative. 
\end{proof}

\bibstyle{alpha}
\bibliography{matroid.bib}

\end{document}